\documentclass{article}

\usepackage[utf8]{inputenc}
\usepackage[T1]{fontenc}
\usepackage[british]{babel}
\usepackage{enumitem}
\usepackage{authblk}

\usepackage{amsmath}
\usepackage{graphicx}
\usepackage{subcaption}
\usepackage[marginal]{footmisc}
\usepackage[all,cmtip]{xy}
\usepackage{amssymb}
\usepackage{caption}
\usepackage{xcolor}
\usepackage{amsthm}
\newtheoremstyle{mystyle_standard}{}{}{}{}{\bfseries}{.}{ }{\thmnumber{#2}.\,\thmname{#1}\thmnote{ (#3)}}
\theoremstyle{mystyle_standard}
\newtheorem{defn}{Definition}[section]

\newtheoremstyle{mystyle_italics}{}{}{\itshape}{}{\bfseries}{.}{ }{\thmnumber{#2}.\,\thmname{#1}\thmnote{ (#3)}}
\theoremstyle{mystyle_italics}
\newtheorem{thm}[defn]{Theorem}

\newtheorem{lem}[defn]{Lemma}

\newtheorem{prop}[defn]{Proposition}

\newtheorem{cor}[defn]{Corollary}
\newtheoremstyle{xxitstyle}{}{}{}{}{\bfseries}{.}{ }{\thmnumber{#2}.\!\!\thmname{#1}\thmnote{ (#3)}}

\newtheoremstyle{mystyle_tt}{}{}{}{}{\itshape}{.}{ }{\thmnumber{#2}.\,\thmname{#1}\thmnote{ (#3)}}
\theoremstyle{mystyle_tt}

\newtheorem{rem}[defn]{Remark}

\theoremstyle{definition}

\newtheorem*{thmnonumber}{Theorem}

\newtheorem*{cornonumber}{Corollary}
\DeclareMathOperator{\dist}{dist}
\DeclareMathOperator{\vol}{vol}

\DeclareMathOperator{\Osc}{Osc}
\DeclareMathOperator{\diam}{diam}

\DeclareMathOperator{\var}{var}
\DeclareMathOperator{\len}{len}

\DeclareMathOperator{\inn}{int}

\DeclareMathOperator{\Cube}{Cube}
\DeclareMathOperator{\BV}{BV}

\DeclareMathOperator{\essinf}{ess\,inf}

\usepackage{mathrsfs}

\usepackage{tikz-cd}
\usepackage{tikz}
\usetikzlibrary{arrows,topaths,calc,patterns,angles,quotes}

\usetikzlibrary{decorations.fractals}
\usetikzlibrary{arrows,calc,patterns}

\newcommand{\neig}{M}
\newcommand{\e}{\epsilon}

\newcommand{\Rohde}{R(p)}
\newcommand{\Kochp}{K(p)}
\newcommand{\Koch}{K}

\renewcommand{\epsilon}{\varepsilon}

\usepackage{hyperref}

\numberwithin{equation}{section}

\allowdisplaybreaks
\usepackage[left=3cm,right=3cm,bottom=2cm,top=2cm]{geometry}
\begin{document}
\title{On bounds for the remainder term of counting functions of the Neumann Laplacian on domains with fractal boundary}

\author[1]{S. Kombrink\footnote{\url{s.kombrink@bham.ac.uk}}}
\author[1]{L. Schmidt\footnote{\url{l.schmidt@bham.ac.uk}\par~\  The second author was supported by EPSRC DTP and the University of Birmingham.}}

\affil[1]{School of Mathematics, University of Birmingham, Edgbaston, Birmingham, B15\,2TT, UK}
\makeatletter
\newcommand{\subjclass}[2][2020]{%
  \let\@oldtitle\@title%
  \gdef\@title{\@oldtitle\footnotetext{#1 \emph{Mathematics subject classification.} #2}}%
}
\newcommand{\keywords}[1]{%
  \let\@@oldtitle\@title%
  \gdef\@title{\@@oldtitle\footnotetext{\emph{Key words and phrases.} #1.}}%
}
\makeatother
\date{}
\maketitle
\vspace{-0.5cm}
\begin{abstract}
 We provide a new constructive method for obtaining explicit remainder estimates of eigenvalue counting functions of Neumann Laplacians on domains with fractal boundary. This is done by establishing estimates for first non-trivial eigenvalues through Rayleigh quotients. 
A main focus lies on domains whose boundary can locally be represented as a limit set of an IFS, with the classic Koch snowflake and certain Rohde snowflakes being prototypical examples, to which the new method is applied. 
 Central to our approach is  the construction of a novel foliation of the domain near its boundary.
 
\end{abstract}
\section{Introduction}
\paragraph{Historic remarks:} In the celebrated 1966 article \cite{kac1966}, Kac asked which geometric information can be inferred from the Laplace spectrum of a domain coining the famous question: \emph{Can one hear the shape of a drum?} This inverse problem, i.\,e.\ computing the \textquotedblleft{}drum\textquotedblright{} given its spectrum, was shown to be ultimately unsolvable in 1992 by Gordon, Webb and Wolpert \cite{gordonwebbwolpert1992} for domains with piecewise smooth boundary and subsequently in 2000 by Chen and Sleemann \cite{chenslee2000} for domains with fractal boundary. 
The focus of the present article lies on the complementary forward problem, namely to decipher the influence of a domain's geometry on its Laplace spectrum. This forward problem has attracted an immense amount of attention throughout the years (see for example \cite{ahrend2009,ivrii2019I} for a summary). Since the Laplace eigenvalue equation can be interpreted physically as the wave equation of a function that is periodic in time, an eigenvalue $\lambda$ can be understood as acoustic frequency of a wave mode. Heuristically, higher harmonics correspond to higher order eigenfunctions displaying similar geometric aspects of their domain. As a consequence, it is particularly interesting how the geometry of a domain dictates the asymptotic behaviour of eigenvalues of the Laplace operator. The idea goes back to 1911 when Weyl established a remarkable law for the eigenvalue counting function $N_D(\Omega,t)$ (i.\,e.\ the number of eigenvalues $\leq t$ for the Dirichlet Laplace problem counted with multiplicity) for open and bounded $\Omega \subset \mathbb{R}^2$ in \cite{weyl1911}, which had a tremendous impact in mathematics:
\begin{align}\label{eq:Weyl_2D}
 N_D(\Omega,t) = \frac{\vol_2(\Omega)}{4\pi}t + o(t)\,, \qquad \text{ as }t \to \infty.
\end{align}
In higher dimensions, the  classic Weyl law for an open $\Omega \subset \mathbb{R}^n$
\begin{align}
 N_D(\Omega,t) = C_W ^{(n)} \vol_n(\Omega) t^{n/2} + o\left(t^{n/2}\right)\,, \qquad \text{ as }t \to \infty \label{eq:weylclassic}
\end{align}
with $C_W ^{(n)}:= 2^{-n}\pi^{-n/2}/\Gamma(1+n/2)$
was shown to hold 
under the minimal hypothesis that $\vol_n(\Omega) < \infty$ in \cite{rozenbljum1972,rozenbljum1976}.  
Based on his result \eqref{eq:Weyl_2D}, Weyl conjectured that for sufficiently regular domains $\Omega \subset \mathbb{R}^n$, the asymptotic behaviour is given by
\begin{align}
 N_D(\Omega,t) 
 = C_W ^{(n)} \vol_n(\Omega) t^{n/2} 
 - \frac{1}{4}C_W ^{(n-1)} \vol_{n-1}(\partial \Omega)t^{(n-1)/2} + o\left( t^{(n-1)/2} \right) \text{ as }t \to \infty. \label{eq:weylconjecture}
\end{align}
A similar conjecture has been formulated for the Neumann counting function; it differs from the Dirichlet-case only by the sign of the second term. These conjectures can be motivated by Weyl's proof of approximating a planar domain $\Omega$ with unions of squares, which suggests that correction terms to the leading order are dictated by the deviation of $\Omega$ from the union of squares and thus by the geometry of the boundary $\partial\Omega$. 
Supposing that $\partial \Omega$ is smooth and additionally assuming that Billiard trajectories are almost never periodic, the Weyl conjecture was proved by Ivrii in 1980 (see \cite{ivrii1980,vassilievsafarov1996,safarovvassiliev1996} and the references therein) for several boundary conditions including Dirichlet and Neumann type. \par
Regularity assumptions in the Weyl conjecture are essential. Indeed, 
if the boundary is irregular such a result 
cannot be expected and is wrong in general. In particular, the remainder term in \eqref{eq:weylconjecture} 
is meaningless when the boundary shows fractal structure with Hausdorff-dimension $\dim_H \partial\Omega > n - 1$ as then $\vol_{n-1}\partial \Omega = \infty$. 
Notice that under Ivrii's conditions, the exponent in the second term coincides with half the topological dimension $\dim (\partial \Omega)/2 = (\dim(\Omega)-1)/2$ of the boundary $\partial\Omega$ 
and that the associated coefficient 
is linked to the domain's surface area.
It is therefore natural to expect analogue substitutes of these quantities in case of non-integer dimensions. \par
Motivated by physical observations, in a first step in this direction, Berry conjectured that the remainder term would be linked to the Hausdorff dimension and the Hausdorff measure of the boundary in \cite{berry1979,berry1980} (see also \cite{la1993}). However, this quickly turned out to be false as a number of counterexamples were found in \cite{brca1986,lapo1993} leading to the formulation of the modified Weyl-Berry conjecture in \cite{la1991}. The modified Weyl-Berry conjecture links the power law of the remainder term in the eigenvalue counting function to the Minkowski dimension and its coefficient to the Minkowski content (see Sec.~\ref{sec:defminkdim}).
 In case of Dirichlet boundary conditions, Lapidus and Pomerance verified the modified Weyl-Berry conjecture under the conditions that $n=1$ and that  $\partial\Omega$ is Minkowski measurable with Minkowski dimension $\delta \in (0,1)$ and Minkowski content $\mathcal M_{\delta}(\partial\Omega)\in(0,\infty)$ by proving $N_D(\Omega,t)=C_W ^{(1)}\vol_1(\Omega) t^{1/2}-c_{\delta}\mathcal M_{\delta}(\partial \Omega)t^{\delta/2}+o(t^{\delta/2})$ as $t\to\infty$. Moreover, they showed that the analogue is incorrect in higher dimensions in general \cite{lapo1993,lapo1996}. Here, $c_{\delta}$ denotes a constant only depending on $\delta$. 
 In arbitrary dimensions $n$, Lapidus showed an asymptotic law
 $N_N(\Omega,t) = C_W ^{(n)}\vol_n(\Omega) t^{n/2} + \mathcal{O}\left(t^{\delta/2}\right)$
 for domains $\Omega$ with an extension property which ensures that the essential spectrum is empty, and under the conditions that $\delta\in(n-1,n]$ and that the upper Minkowski content of $\partial \Omega$ exists in \cite{la1991}.
 In contrast, the existence of the upper inner Minkowski content was shown to be sufficient for estimates of the Dirichlet counting function in the same article.
 Of particular relevance to the present article is the work \cite{NetrusovSafarov2005}, where
 an exact rather than an asymptotic bound on $N_N(\Omega,t)-C_W ^{(n)}\vol_n(\Omega) t^{n/2}$ was obtained for domains whose boundary locally is a graph of a $BV$-function.
\paragraph{Overview of main results:} Motivated by the fact that among counting functions for mixed boundary conditions the Neumann case provides an upper bound, the present article focuses on remainder estimates of Neumann eigenvalue counting functions $N_N(\Omega,t)$.
To obtain these, we introduce the concept of well-covered domains (see Def.~\ref{def:wellcovered}). For such domains $\Omega$, a
main result, stated in the below theorem provides 
an explicit upper bound of $N_N(\Omega,t)-C_W ^{(n)}\vol_n (\Omega) t^{n/2}$ for all $t\geq t_0$ for a determined $t_0$.
\begin{thmnonumber}[Thm.~\ref{thm:satz1}]
Let $\Omega \subset \mathbb{R}^n$ be a domain whose boundary $\partial\Omega$ has  
upper inner Minkowski dimension $\delta > n-1$ (see Sec.~\ref{sec:defminkdim}). Suppose that $\Omega$ is well-covered (see Def.~\ref{def:wellcovered}), meaning that for all sufficiently small $\epsilon>0$ the set $\Omega_{-\epsilon}:=\{ x\in \Omega\ :\ \dist(x,\partial\Omega)< \epsilon\}$ admits a cover by domains with uniformly comparable diameter that satisfies the following. (i) The cardinality of the cover scales like $\epsilon^{-\delta}$ and (ii) each covering domain admits a well-behaved foliation (see Def.~\ref{def:wellfoliated}). 
Then there is an explicitly determined $M_{\Omega}$ and $t_0\in\mathbb R$ such that
 \begin{align}
    N_N(\Omega,t) - C_W ^{(n)} \vol_n(\Omega)t^{n/2} \leq   M_{\Omega} t^{\delta/2}\label{eq:intro:glsatz1}
\end{align}
for all $t\geq t_0$.
\end{thmnonumber}
Examples of domains with fractal boundary to which the above theorem applies are Koch and certain Rohde snowflakes (see Sec.~\ref{subsec:application}). Rhode snowflakes are particularly important domains as, up to bi-Lipschitz transformation, every planar quasidisk (i.\,e.\ a domain that is bounded by the image of a circle under a quasi-conformal map) can be realised as such a snowflake (see \cite{rohde2001}).
In this context, it is worth noting that bi-Lipschitz maps preserve upper inner Minkowski dimension, the property of being well-covered and hence all prerequisites of Thm.~\ref{thm:satz1}.
Another significant subclass of domains $\Omega$ for which Thm.~\ref{thm:satz1} holds is the class of domains whose boundary locally is a graph of a $BV$-function (see Sec.~\ref{subsec:bdvar}). For this subclass an estimate of the form \eqref{eq:intro:glsatz1} is presented in \cite{NetrusovSafarov2005}. Indeed, the work of \cite{NetrusovSafarov2005} gave the motivation for developing the novel concept of well-covered domains in order to achieve the estimate \eqref{eq:intro:glsatz1} for snowflake domains. 
For showing that a snowflake domain is well-covered we develop a new way of explicitly constructing a  foliation of the domains covering the inner $\epsilon$-neighbourhood $\Omega_{-\epsilon}$, 
see Sec.~\ref{subsec:foliation}, addressing (ii) in the above theorem. Note that the existence of the upper Minkowski content is not necessary in Thm.~\ref{thm:satz1}, but that the assumptions of Thm.~\ref{thm:satz1} imply the existence of the upper inner Minkowski content.\par 
Under the assumptions of Thm.~\ref{thm:satz1}, we can also treat the case $\delta=n-1$, in which we obtain $N_N(\Omega,t) - C_W ^{(n)} \vol_n(\Omega)t^{n/2} =\mathcal O \left( t^{(n-1)/2}\right)$ as $t\to\infty$. This addresses a question in \cite{la1993}, where the weaker estimate $N_N(\Omega,t) - C_W ^{(n)} \vol_n(\Omega)t^{n/2} =\mathcal O \left( t^{(n-1)/2}\log(t)\right)$ is given, see Rem.~\ref{rem:countingfunctions}\ref{item:allelambda}.
\begin{rem}\label{rem:spray}
Having the estimate \eqref{eq:intro:glsatz1} for all $t\geq t_0$ with explicitly determined $M_{\Omega}$ and $t_0$ rather than asymptotically is key for obtaining Dirichlet Weyl laws for self-similar sprays $U$, whose generator is a finite union of pairwise disjoint well-covered domains. (Loosely, a self-similar spray is a countable union of scaled copies of its generator.)
Indeed, in the companion paper \cite{DOKUMENT2} by the authors, Thm.~\ref{thm:satz1}, in conjunction with an estimate by van den Berg and Lianantonakis from \cite{vandenBerg2001}, is used as a pivotal tool to achieve the asymptotics
\begin{equation}
    N_D(U,t)
    = C_W ^{(2)} \vol_2(U)t + \sum_{j\in J} A_j(t) t^j + \mathcal O(t^{\frac{\log 4}{2\log 3}}) 
    \label{eq:companion}
\end{equation}
for a family of self-similar sprays $U$ whose generator is a finite union of pairwise disjoint Koch snowflakes. Here, the $A_j$ are bounded and $J$ is of finite cardinality. 
Cor.~\ref{cor:RundQ}, which allows to combine Thm.~\ref{thm:satz1} with the estimate from \cite{vandenBerg2001}, hinges on common conditions for the existence of estimates of Neumann and Dirichlet counting functions.
Depending on the structure of $U$, the asymptotic \eqref{eq:companion} can exhibit several asymptotic terms of faster growth than the error term $O(t^{\frac{\log 4}{2\log 3}})$. Explicit examples for $U$ for which \eqref{eq:companion} gives two, three and four summands of faster growth than $t^{\frac{\log 4}{2\log 3}}$ are presented in \cite{DOKUMENT2}, and hints given for which sprays Thm.~\ref{thm:satz1} leads to even more asymptotic terms of faster growth than the error term.
Notably, the exponents in \eqref{eq:companion}, i.\,e.\ the elements of $J$, are shown to be in one-to-one correspondence with the exponents in the volume expansion of $\e\mapsto\vol_2\left( U_{-\e} \right)$, providing valuable insights as to how the fine geometric structure of a self-similar spray influences the distribution of the Dirichlet Laplace eigenvalues.
\end{rem}
The rough outline for obtaining the explicit bound on $N_N(\Omega,t)$ for well-covered domains $\Omega$ in Thm.~\ref{thm:satz1} is to partition $\Omega$ into $\Omega_{-\epsilon}$ and $\Omega\setminus \Omega_{-\epsilon}$.
We use a Whitney cover of $\Omega$ and restrict it to a cover of $\Omega \backslash \Omega_{-\epsilon}$ by Whitney cubes $\{Q\,:\, Q\in \mathcal{W}_{\e}\}$. For this cover we show that $\vol_{n-1}\left(\partial\overline{\bigcup_{Q \in \mathcal{W}_{\e} } Q }\right)=\mathcal{O}\left(\epsilon^{(n-1)-\delta}\right)$ as $\e\to 0$ (see Prop.~\ref{prop:whitneycardinality}). This is used for an asymptotic law of $N_N\left( \overline{\bigcup_{Q \in \mathcal{W}_\epsilon} Q} , t \right)$.
For handling $\Omega_{-\epsilon}$ we use that $\Omega$ is well-covered, denoting the covering domains of $\Omega_{-\e}$ by $D_i^{\e}$. The well-behaved foliation of $D_i^{\e}$ is used to show a Poincar\'{e}-Wirtinger-like inequality for all non-constant $u \in H^1(D_i^{\e})$ by separating $D_i^{\e} = E \cup D_i^{\e}\backslash E$ where $E$ has a known first non-trivial eigenvalue $\lambda_2 ^N(E)$. The inverse of the first non-trivial eigenvalue of a domain $D_i^{\e}$ with respect to Neumann boundary conditions, denoted by $\lambda_2 ^N(D_i^{\e})$, is the best Poincar\'{e} constant for non-constant functions $u \in H^1(D_i^{\e})$. We use a version of the variational technique introduced in \cite{NetrusovSafarov2005} and thus avoid the explicit need for existence of bounded extensions $W^{1,2}(\Omega) \hookrightarrow W^{1,2}(\mathbb{R}^n)$ of the Sobolev spaces involved.
 \begin{cornonumber}[cf.\ Cor.~\ref{cor:ewestimate}]
 Let $\Omega$ be well-covered and let $\{D^\epsilon _i \}_{i \in I_\epsilon}$ be a cover of $\Omega_{-\epsilon}$ consisting of well-foliated domains as in Def.~\ref{def:wellcovered}. Then by Lem.~\ref{lem:ewestimate}, there exists $C_1(\Omega_{-\epsilon})$ with $\inf_{\epsilon>0} C_1(\Omega_{-\epsilon})>0$ such that
 \begin{align*}
  \lambda_2 ^N(D^\epsilon _i) \geq C_1(\Omega_{-\epsilon}) \epsilon^{-2}.
 \end{align*}
\end{cornonumber}

Notice that the uniform comparable geometry of the covering domains of $\Omega_{-\epsilon}$ as used in Thm.~\ref{thm:satz1} above is precisely the condition to ensure a scaling behaviour $\lambda_2 ^N(D_i ^\epsilon) \gtrsim \diam(D_i^{\e})^{-2}$. In particular for a family of such domains $\{D_i ^\epsilon\}_{i \in I_\epsilon}$, $\sum_{i \in I_\epsilon} N_N(D_i ^\epsilon,t) = \#\{I_\epsilon\}$ whenever $\e$ is small enough. 
 Choosing $\epsilon$ depending on $t$, these estimates allow for sufficient control of $N_N(\Omega,t) \leq \sum_{\{i \in I_\epsilon\}}N_N(D_i ^\epsilon,t) + \sum_{Q \in \mathcal{W}_\epsilon} N_N(Q,t)$ to prove Thm.~\ref{thm:satz1}. 
\paragraph{Outline:} The present article is structured as follows: We first introduce the relevant notions of Laplace eigenvalue counting functions, Minkowski content and Whitney covers in Sec.~\ref{sec:preliminaries}. In Sec.~\ref{subsec:foliation} we show how one can construct families of paths ending at the boundary of a domain $D$ and show how this structure on $D$ gives rise to a lower bound for the first non-trivial Neumann eigenvalue $\lambda_2 ^N(D)$ of the form $\lambda_2 ^N(D) \gtrsim \diam(D)^{-2}$ (see Lem.~\ref{lem:ewestimate}). In Sec.~\ref{subsec:bounds}, we use this bound to obtain explicit and asymptotic estimates for the remainder term of counting functions in Thm.~\ref{thm:satz1}, providing an estimate for an $M_{\Omega}$ such that \eqref{eq:intro:glsatz1} holds 
for all sufficiently large $t$.  Sec.~\ref{subsec:bdvar} is devoted to proving that sets of bounded variation are well-covered domains. Finally, in Sec.~\ref{subsec:application} we apply Thm.~\ref{thm:satz1} to the classic Koch snowflake and certain Rohde snowflakes, providing explicit constructions and results for these cases. 
In the final Sec.~\ref{sec:constants}, important constants, that are used throughout the article, are collected.
 \section{Preliminaries}\label{sec:preliminaries} 
 Let $\Omega\subset\mathbb R^n$ be a bounded domain, i.\,e.\ a bounded open connected subset of $\mathbb{R}^n$ and denote its boundary by $\partial \Omega$. 
 By $H^1(\Omega)=W^{1,2}(\Omega)$ we denote the usual Sobolev space, i.\,e.\ the set of all $u \in L^2(\Omega)$ with a weak derivative $\nabla u \in L^2(\Omega)$. We equip $H^1(\Omega)$ with the usual inner product $(u,v)_{H^1(\Omega)} := (u,v)_{L^2(\Omega)} + (\nabla u,\nabla v)_{L^2(\Omega)}$ so that $H^1(\Omega)$ becomes a Hilbert space. Further, $H^1 _0 (\Omega)\subset H^1(\Omega)$ shall denote the closure of the set of compactly supported $C^\infty(\Omega)$-functions in $H^1(\Omega)$.
 On $H^1(\Omega)$, resp.\ $H^1 _0(\Omega)$, we consider the Laplacian $\Delta := \sum_{i=1} ^n \partial_i ^2$ and focus on the eigenvalue equation of $-\Delta$ subject to Neumann \eqref{eq:ewgl} or Dirichlet \eqref{eq:ewgld} boundary conditions.
 
 \begin{minipage}{0.4\textwidth}
\begin{align}
\begin{cases}
    -\Delta u = \lambda u &\text{ in } \Omega\\
    \frac{\partial u}{\partial \mathbf{n}} = 0 & \text{ on } \partial\Omega
\end{cases} \label{eq:ewgl}
\end{align}
 \end{minipage}
 \hspace{0.1\textwidth}
 \begin{minipage}{0.4\textwidth}
\begin{align}
\begin{cases}
    -\Delta u = \lambda u & \text{ in } \Omega\\
    u = 0 & \text{ on } \partial\Omega
\end{cases} \label{eq:ewgld}
\end{align}
 \end{minipage}~\\
 
\noindent{where $\mathbf{n}$ denotes the exterior normal to $\partial \Omega$.}
The variational formulation of the problem \eqref{eq:ewgl} is stated as follows: Find $u \in H^1(\Omega)$ such that $(\nabla u,\nabla v)_{L^2{(\Omega)}} = \lambda (u,v)_{L^2{(\Omega)}}$ for all $v \in H^1(\Omega)$. Note that the space in which this problem is studied dictates the boundary condition and that the variational formulation of \eqref{eq:ewgld} is: Find $u \in H^1_0(\Omega)$ such that $(\nabla u,\nabla v)_{L^2{(\Omega)}} = \lambda (u,v)_{L^2{(\Omega)}}$ for all $v \in H^1_0(\Omega)$. 
Replacing $H^1(\Omega)$ resp.\ $H^1_0(\Omega)$ with any other closed space $V$ satisfying $H^1 _0(\Omega) \subset V \subset H^1(\Omega)$ gives rise to variational problems with more general boundary conditions.
\par
The corresponding non-negative spectrum will be denoted by $\sigma(-\Delta)$ and the essential spectrum by $\sigma_{\text{ess}}(-\Delta)$. 
 Supposing that the non-essential spectrum of the eigenvalue problem \eqref{eq:ewgl} contains at least two isolated points $0 = \lambda_1 ^N(\Omega) <\lambda_2 ^N(\Omega) < \inf \sigma_{\text{ess}}(-\Delta)$, it is well known that
  \begin{align}
  \lambda_2 ^N = \inf_{\substack{v \in H ^1(\Omega)\backslash\{0\}:\\ \int_{\Omega} vdx=0}} \frac{\|\nabla v\|^2 _{L^2(\Omega)}}{\| v \|^2 _{L^2(\Omega)}}. \label{eq:rayleigh}
 \end{align}
 Writing $1^\perp$ for the set of non-constant functions, we then have:
    \begin{cor}[cf.\ \cite{NetrusovSafarov2005}]\label{cor:minmax2}
   Let $\lambda_2 ^N(\Omega)$ be the first non-trivial Neumann eigenvalue of $\Omega$. Then
   \begin{align*}
    \lambda_2 ^N(\Omega) \geq L^{-1} \Leftrightarrow \left(\| u \|^2 _{L^2(\Omega)} \leq L \| \nabla u \|^2 _{L^2(\Omega)} \,\,\forall u \in H^1 (\Omega) \cap 1^\perp \right).
   \end{align*}
  \end{cor}
  This classic variational result will be crucial in proving lower bounds for the first non-trivial eigenvalue under Neumann conditions (see Lemma~\ref{lem:ewestimate}). The analogous statements to \eqref{eq:rayleigh} and Cor.~\ref{cor:minmax2} hold true for other boundary conditions by replacing $H^1(\Omega)$ 
  with any closed space $V$ satisfying $H^1 _0(\Omega) \subset V \subset H^1(\Omega)$ in both statements, and by replacing $1^\perp$ with 'orthogonal to the eigenfunction to $\lambda_1$' (see for example \cite{teplyaev2022,la1991}).
  \par
 In order to define a counting function $N(\Omega,t) := \#\{\lambda \in \sigma(-\Delta):\lambda \leq t\}$, it is necessary that $\sigma(-\Delta)$ is discrete with the only accumulation point being at $\infty$. While this is always satisfied in case of Dirichlet boundary conditions, there are several occasions where this may fail to be true under Neumann boundary conditions as the essential spectrum can be non-empty in this case. Indeed, it was shown in \cite{hesesi1990} that any closed subset of $\mathbb{R}_{\geq 0}$ can be realised as the essential spectrum of a Laplacian on a domain $\subset \mathbb{R}^2$ subject to Neumann boundary conditions. On the other side, several criteria have been found which ensure that the essential spectrum is empty, see for example \cite{netrusov2007}. In \cite{rellich1948} Rellich showed that the Neumann spectrum is discrete whenever the inclusion $\iota:H^2(\Omega) \hookrightarrow L^2(\Omega)$ is compact.\footnote{The original German publication uses the concept of completely continuous operators (\glqq{}vollstetige Operatoren\grqq{}). Whenever the domain is a Hilbert space, this notion coincides with the usual notion of a compact operator.} Maz'ya showed in \cite{ma1968} that $\iota$ is compact iff $\lim_{M \to 0} \inf_{t\in (0,M)} t^{-1} \inf_{\substack{\vol_n(F_1) \geq t\\ \vol_n(F_2) \leq M}} c(F_1,F_2) = 0$. Here, $c(F_1,F_2):=\inf_{f \in V(F_1,F_2)} \int_{\Omega} |\nabla f|^2dx$ is called \emph{relative capacity} where $V(F_1,F_2) :=\{f \in C^\infty(\Omega):f|_{F_1}=1, f|_{F_2}=0\}$ for any two disjoint $F_1,F_2 \subset \Omega$ that are closed in $\Omega$ giving possible descriptions of such domains.\par
 Any domain $\Omega$ that admits a bounded linear extension operator $E:H^1(\Omega) \to H^1(\mathbb{R}^n)$ (called an \emph{extension domain}) has a discrete Neumann spectrum, as the Neumann counting function exists in this case (see e.\,g.\ \cite{la1991}). In the planar case, a simply connected domain $\Omega$ is an extension domain iff $\partial \Omega$ is a \emph{quasicircle}, as was shown in \cite{vodo1979,jones1981}. Recall that a quasicircle is the image of the planar unit circle under a quasiconformal map. 
 It can be shown (see \cite{ahlfors1963} or \cite{mazya2011} and the references therein) that a Jordan curve $\gamma \subset \mathbb{R}^2$ is quasicircular if there is some $c \in \mathbb{R}$ such that for any two points $a,b \in \gamma$ one has $\diam \Gamma \leq c|a-b|$, where $\Gamma \subset \gamma \backslash \{a,b\}$ is the component with smaller diameter. Moreover, Rohde has shown in \cite{rohde2001} that any quasicircle is a snowflake up to a bi-Lipschitz map, in particular showing that any possible snowflake is also a quasicircle, see Rem.~\ref{rem:Rohde}. Higher dimensional domains that are bounded by quasispheres or the  $(\epsilon,\delta)$-domains introduced by Jones in \cite{jones1981} are also extension domains but this set is no longer exhaustive \cite{teplyaev2022}; see also \cite{hajlasz2008} for a complete description of such extension domains and \cite{ahlfors2006,rickman1966} for further characterisation of such maps.
\subsection{Dirichlet-Neumann bracketing} 
  For a domain $\Omega \subset \mathbb{R}^n$ a \emph{volume cover} $\{\Omega_i\}_{i \in I}$ of $\Omega$ consists of at most countably many open sets $\Omega_i \subset \Omega$ with $\vol_n(\Omega) = \vol_n\left( \bigcup_{i \in I}\Omega_i \right)$. We write $N_N(\Omega,t)$ (resp.\ $N_D(\Omega,t)$) for the number of eigenvalues (with multiplicity) of $-\Delta$ on $\Omega$ subject to Neumann (resp.\ Dirichlet) conditions on $\partial \Omega$. In addition to the classic Dirichlet-Neumann bracketing linking the counting functions of Dirichlet and Neumann eigenvalues, we mention a version that allows for non-disjoint covers as long as the elements of the cover do not intersect too often. More precisely one has the following result which also follows from the Min-Max-Principle. For any volume cover $\{\Omega_i\}_{i \in I}$ of $\Omega$, let $\mu:=\sup_{x \in \Omega}\#\{i \in I:x \in \Omega_i\}$ denote its \emph{multiplicity}.
  \begin{prop}[Dirichlet-Neumann bracketing with multiplicity, \cite{NetrusovSafarov2005}]\label{prop:bracketing}
   Let $\{\Omega_i\}_{i \in I}$ be a volume cover of $\Omega$. If the $\Omega_i$ are pairwise disjoint, then
   \begin{align*}
    \sum_{i \in I} N_D(\Omega_i,t) \leq N_D(\Omega,t) \leq N_N(\Omega,t) \leq \sum_{i \in I} N_N(\Omega_i,t)
   \end{align*}
   If the volume cover has finite multiplicity $\mu$, then
   \begin{align*}
    N_{N}(\Omega,t) \leq \sum_{i \in I} N_{N}(\Omega_i,\mu t).
   \end{align*}
  \end{prop}
   \begin{cor}\label{cor:RundQ}
  Let $\{\Omega_n\}$ be a finite volume cover of a domain $\Omega$. Define $Q(\Omega',t) := N_N(\Omega',t) - N_D(\Omega',t)$ for any bounded open set $\Omega'$ and let $R_D(t):=N_D(\Omega,t)-\sum_{n \in \mathbb{N}} N_D( \Omega_n,t)$. Then
  \begin{align*}
   R_D(t) \leq - Q(\Omega,t) + \sum_{n \in \mathbb{N}} Q(\Omega_n,t).
  \end{align*}
 \end{cor}
 \begin{proof}
  Prop.~\ref{prop:bracketing} shows that $Q(\Omega',t)\geq 0$ and $N_N(\Omega,t) - \sum_n N_N( \Omega_n,t) \leq 0$. Moreover,
  \begin{align*}
 R_{D} (t) &= N_D (\Omega,t)-\sum_{n} N_D  (\Omega_n,t) =N_N (\Omega,t)-\sum_n N_D (\Omega_n,t)- Q(\Omega,t)\\
 &\leq \sum_n N_N (\Omega_n,t)- \sum_n N_D (\Omega_n,t)- Q(\Omega,t)=\sum_n Q(\Omega_n, t) -Q(\Omega,t).\qedhere
\end{align*}
 \end{proof}
 A simple estimate for $Q$ was found in \cite{SafarovFilonov2010} based on estimates on the Neumann counting function.
\subsection{Minkowski content and dimension}\label{sec:defminkdim} 
 Let $X$ be any non-empty open bounded subset of $\mathbb{R}^n$ with boundary $\partial X$. We write
 \begin{align*}
      X_\epsilon := \{x \in \mathbb{R}^n:\dist(x, X)<\epsilon\}\,\text{ and }X_{-\epsilon} := \{x \in X:\dist(x,\partial X)<\epsilon\} = (\partial X)_\epsilon \cap X.
  \end{align*}
 Then for any $s>0$ one defines the corresponding \emph{upper inner $s$-Minkowski content} as
  \begin{align*}
   \overline{\mathcal{M}} _s(\partial X,X) := \limsup_{\epsilon \searrow 0} \underbrace{\vol_n (X_{-\epsilon})\epsilon^{-(n-s)}}_{=:\overline{\mathcal{M}}_s ^\epsilon(\partial X,X)}
  \end{align*}
  whenever this exists. We call $\partial X$ \emph{upper inner Minkowski measurable} whenever $\overline{\mathcal{M}}_s(\partial X,X) \in (0,\infty)$ for some $s$. We observe that (similar to the case of Hausdorff dimensions), one has for $s,t>0$
  \begin{align*}
   \overline{\mathcal{M}}_t ^\epsilon(\partial X,X) = \epsilon^{t-s}\overline{\mathcal{M}}_s ^\epsilon(\partial X,X).
  \end{align*}
  It follows that $s \mapsto \overline{\mathcal{M}}_s (\partial X,X)$ has a critical \glqq{}jump\grqq{} from $+\infty$ to $0$ at exactly one $s \in \mathbb{R}_{\geq 0}$. This is defined as the \emph{upper inner Minkowski dimension}:
  \begin{align*}
   \overline{\dim_M} (\partial X,X) := \inf \left\{s: \overline{\mathcal{M}}_s (\partial X,X) = 0 \right\}. 
  \end{align*}
  Whenever $\partial X$ is upper inner Minkowski measurable, the upper inner $s$-Minkowski content for $s=\overline{\dim_M}(\partial X,X)$ is called the \emph{upper inner Minkowski content of $\partial X$} and is denoted by $\overline{\mathcal{M}}(\partial X,X)$. The classic upper (i.\,e.\ non-inner) Minkowski content and dimension are obtained by replacing $X_{-\epsilon}$ with $X_\epsilon$ in this definition.
  \subsection[Regularisation of Omega - Omega-epsilon by Whitney covers]{Regularisation of $\boldsymbol{\Omega \backslash \Omega_{-\epsilon}}$ by Whitney covers}\label{subsec:whitney}
  Whitney covers are a common construction used to prove estimates on counting functions by approximating a domain by a discrete union of cubes. Let $\Omega \subset \mathbb{R}^n$ be an arbitrary bounded domain. A Whitney cover of $\Omega$ is a volume cover of $\Omega$ by cubes of different sizes such that a cube containing $x \in \Omega$ has a diameter that is uniformly comparable to $\dist(x,\partial \Omega)$. The construction of such a cover is well known and sometimes attributed to Whitney, who used it to study extensions of functions in \cite{whitney1934}. However, the basic idea behind it had in fact already been published by Courant and Hilbert in \cite[pp.~355-356]{CourantHilbert1924} in 1924 when Whitney was only 17.  
  For the sake of completeness we include the construction from \cite[Chap.~VI§1]{stein1970}, as the notation introduced here will be used in later results: Consider the nested family of lattices $\left\{2^{-k}\mathbb{Z}^n\right\}_{k \in \mathbb{Z}}$. To each point $p=(p_1,\ldots,p_n)$ in the lattice $2^{-k}\mathbb{Z}^n$ we take the open cubes $\prod_{i=1} ^n \left(p_i,p_i+2^{-k}\right)$ with diagonal of length $2^{-k}\sqrt{n}$ in the positive quadrant of $p$. The set of these cubes is denoted by $\Cube(2^{-k}\mathbb{Z}^n)$. We then slice up $\Omega$ into sectors and cover them individually:
   \begin{align}
    \mathcal{W}_k &:= \left\{Q \in \Cube(2^{-k}\mathbb{Z}^n):Q \cap \left( \overline{\Omega_{-2^{-k+2}\sqrt{n}}} \backslash \Omega_{-2^{-k+1}\sqrt{n}} \right)\neq \emptyset \right\} \label{eq:constructionwhitneylocal} \\
    \mathcal{W}' &:= \bigcup_{k \in \mathbb{Z}} \mathcal{W}_k. \nonumber
   \end{align}
   Since $\Omega = \bigcup_{k \in \mathbb{Z}} \left(\overline{\Omega_{-2^{-k+2}\sqrt{n}}} \backslash \Omega_{-2^{-k+1}\sqrt{n}}\right)$, it follows that $\vol_n\Omega = \vol_n\left(\bigcup_{Q \in \mathcal{W}'} Q\right)$. For the following, let $Q \in \mathcal{W}'$ and let $k \in \mathbb{Z}$ be so that $Q \in \mathcal{W}_k$. Then by definition, $\diam(Q)=2^{-k}\sqrt{n}$ and there exists $x \in Q \cap \left( \overline{\Omega_{-2^{-k+2}\sqrt{n}}} \backslash \Omega_{-2^{-k+1}\sqrt{n}} \right)$. For this $x$ we have $2^{-k+1}\sqrt{n} \leq \dist(x,\partial \Omega) \leq 2^{-k+2}\sqrt{n}$. Furthermore, $\dist(Q,\partial \Omega) \leq 2^{-k+2}\sqrt{n} = 4\diam(Q)$. Moreover, we have $\diam(Q) = 2^{-k+1}\sqrt{n}-2^{-k}\sqrt{n} \leq \dist(x,\partial \Omega)-\diam(Q) \leq \dist(Q,\partial \Omega)$. However, $\mathcal{W}'$ might contain overlapping cubes. To rule out such cubes, observe that any two intersecting cubes are nested. Suppose two cubes $Q,Q'$ intersect with $Q \subset Q'$. Then by the above
   \begin{align*}
    \diam(Q') \leq \dist(Q',\partial \Omega) \leq \dist(Q,\partial \Omega) \leq 4\diam(Q).
   \end{align*}
   This shows that to any nested sequence of intersecting cubes in $\mathcal{W}'$ there is a unique maximal cube containing all others. We define $\mathcal{W} \subset \mathcal{W}'$ as the set of all those maximal cubes, which then are pairwise disjoint and $\mathcal{W}$ inherits the other properties from $\mathcal{W}'$. Such a volume cover will be called a \emph{Whitney cover}. This proves:
  \begin{lem}[Existence of Whitney covers, \cite{stein1970}]\label{lem:exwhitney}
   Let $\Omega \subset \mathbb{R}^n$ be a domain. Then there is a volume cover of $\Omega$ consisting of pairwise disjoint cubes $\{Q\}_{Q \in \mathcal{W}}$ and
  \begin{align}
   \diam(Q) \leq \dist(Q,\partial \Omega) \leq 4\diam(Q) \qquad \forall Q \in \mathcal{W}. \label{eq:whitney}
  \end{align}
  \end{lem}
  Based on this, we have an immediate estimate on the number of cubes of a certain size in a Whitney cover.
   \begin{prop}[Cardinality of slices of Whitney covers]\label{prop:whitneycardinality}
    Let $\mathcal{W}$ be a Whitney cover of a domain $\Omega \subset \mathbb{R}^n$, as constructed above. Suppose that there exists $\delta \in [n-1,n)$ such that $\overline{\mathcal{M}}_\delta(\partial \Omega,\Omega) \in (0,\infty)$. Then there is $\mathfrak{M}_\Omega \in \mathbb{R}$ such that $\#\mathcal{W}_k \leq \mathfrak{M}_\Omega 2^{k \delta}$.
    \end{prop}
    Similar results concerning the Hausdorff measure were obtained for example by Käenmäki, Lehrbäck and Vuorinen in \cite{antti2013}.
    \begin{proof}
    Let $d_k:=2^{-k}\sqrt{n}$. By construction of $\mathcal{W}_k$ as in \eqref{eq:constructionwhitneylocal},
     \begin{align*}
      \# \mathcal{W}_k &\leq \frac{\vol_n \left( \Omega_{-\left( 4d_k+d_k \right)} \right) - \vol_n \left( \Omega_{-\left( 2d_k -d_k \right)} \right)}{2^{-kn}} 
      \leq \frac{ \vol_n (\Omega_{-5\cdot 2^{-k}\sqrt{n}}) }{2^{-kn}}.
     \end{align*}
     Defining $\epsilon'(\epsilon):=\epsilon^{\delta-n}\vol_n(\Omega_{-\epsilon})/\overline{\mathcal{M}}_\delta(\partial \Omega,\Omega)-1$, we have
     \begin{equation*}
      \#\mathcal{W}_k \leq \overline{\mathcal{M}}_\delta(\partial \Omega,\Omega) \left( 1 + \epsilon'\left( 5\sqrt{n} 2^{-k} \right) \right) \left( 5\sqrt{n} \right)^{n-\delta} 2^{k\delta}.
     \end{equation*}
      Then for $\mathfrak{M}_\Omega:= \sup_{k} \overline{\mathcal{M}}_\delta(\partial \Omega,\Omega) \left( 1 + \epsilon'\left( 5\sqrt{n} 2^{-k} \right) \right) \left( 5\sqrt{n} \right)^{n-\delta}$ and the assertion follows from \linebreak$\limsup_{\epsilon \to 0} \epsilon'(\epsilon)=0$.
    \end{proof}
    \begin{prop}\label{prop:whitneyumfang}
        Let $\Omega \subset \mathbb{R}^n$ be bounded with $\delta:=\overline{\dim_M}(\partial \Omega,\Omega)<\infty$ and upper inner Minkowski content $\overline{\mathcal{M}}_\delta(\partial \Omega,\Omega)\in (0,\infty)$. Let $\mathcal{W}$ be a Whitney cover of $\Omega$ and for any $\epsilon > 0$, let
        \begin{align*}
            \mathcal{W}_\epsilon := \left\{ Q \in \mathcal{W}: Q \cap \left( \Omega \backslash \Omega_{-\epsilon} \right) \neq \emptyset \right\},
        \end{align*}
        in other words, $\mathcal{W}_\epsilon$ is the smallest collection of Whitney cubes in $\mathcal W$ that covers $\{x \in \Omega:\dist(x,\partial \Omega) \geq \epsilon\}$. Then there is $A_\Omega \in \mathbb{R}$ such that
        \begin{align*}
            \vol_{n-1}  \partial \left( \overline{\bigcup_{Q \in \mathcal{W}_\epsilon} Q} \right) \leq A_\Omega \cdot \epsilon^{(n-1)-\delta}.
        \end{align*}
    \end{prop}
    \begin{proof}
        Since the closure of the outermost cubes in $\mathcal{W}_\epsilon$ intersect $\partial \Omega_{-\epsilon}$, we have $\vol_{n-1} \partial\left(\overline{\bigcup_{Q \in \mathcal{W}_\epsilon} Q }\right) \leq \sum_{Q\in \mathcal{W}_{\epsilon}:\overline{Q} \cap \partial\Omega_{-\epsilon} \neq \emptyset} \vol_{n-1} \partial Q$. Suppose that $\overline{Q}$ has non-empty intersection with $\partial \Omega_{-\epsilon}$. Then there exists $x \in \overline{Q}$ such that $\dist(x,\partial \Omega) = \epsilon$. Now with \eqref{eq:whitney} and $\dist(Q,\partial \Omega) +\diam(Q) \geq \sup_{x \in Q}\dist(x,\partial \Omega) \geq \epsilon \geq \dist(Q,\partial \Omega)$, we observe that
        \begin{align*}
            5 \diam(Q) \geq \epsilon \geq \diam(Q) \geq \frac{\epsilon}{5}.
        \end{align*}
        With $k_{\min} := \min\{k \in \mathbb{Z}:2^{-k}\sqrt{n} \leq \epsilon\}$, the diameter of such a cube is therefore restricted to one of at most three possible sizes, $\diam(Q) \in \{ 2^{-k_{\min}}\sqrt{n},2^{-(k_{\min}+1)}\sqrt{n},2^{-(k_{\min}+2)}\sqrt{n} \}.$
        Hence, an upper bound of the circumference of $\overline{\bigcup_{Q \in \mathcal{W}_\epsilon} Q}$ is
        \begin{align*}
            \vol_{n-1}  \partial \left(\overline{\bigcup_{Q \in \mathcal{W}_\epsilon} Q }\right)  &\leq \sum_{k=k_{\min}} ^{k_{\min}+2} \sum_{Q \in \mathcal{W}_k} \vol_{n-1}\partial Q = 2n \sum_{k=k_{\min}} ^{k_{\min}+2} \#\mathcal{W}_k \cdot 2^{-k(n-1)}\\
            &\leq 2n\mathfrak{M}_\Omega \sum_{k=0} ^{2} 2^{-k((n-1)-\delta)} \cdot 2^{-k_{\min}((n-1)-\delta)} \leq A_\Omega \epsilon^{(n-1)-\delta},
        \end{align*}
        where $A_\Omega := 2n(2\sqrt{n})^{\delta-(n-1)} \mathfrak{M}_\Omega \sum_{k=0} ^{2} 2^{-k((n-1)-\delta)}$ since one has $2 \cdot 2^{-k_{\min}}\sqrt{n} \geq \epsilon$ and $n-1 \leq \delta$.
    \end{proof}
    In particular in the regular case (i.\,e.\ when $\delta = n-1$), Prop.~\ref{prop:whitneyumfang} shows that the circumference of the closure of the union of all large enough cubes of a Whitney cover is bounded. 
    \begin{prop}\label{prop:whitneybillard}
        In the setting of Prop.~\ref{prop:whitneyumfang} almost all Billard trajectories in $\overline{ \bigcup_{Q \in \mathcal{W}_\epsilon} Q}$ are non-periodic.
  \end{prop}
    \begin{proof}
        We may assume that $\mathcal{W}_\epsilon$ consists only of cubes that have at least one face shared with another cube in $\mathcal{W}_\epsilon$, as the statement follows from known results for isolated cubes. Then there is $k_{\max} \in \mathbb{Z}$ such that $\mathcal{W}_\epsilon \subset \bigcup_{k \leq k_{\max}} \mathcal{W}_k$ and all cubes in $\mathcal{W}_\epsilon$ have vertices in the $2^{-k_{\max}}\mathbb{Z}^n$-lattice. Let $x \in \partial \overline{ \bigcup_{Q \in \mathcal{W}_\epsilon} Q}$ and $v$ be any direction of a Billard. Then, up to lattice symmetry,  the set of reflection points of such a trajectory in $ \overline{ \bigcup_{Q \in \mathcal{W}_\epsilon} Q}$ is contained in $\{x+tv:t \in \mathbb{R}\backslash\{0\}\}$. Up to lattice symmetry, this set contains the origin $x$ iff $v$ is rational up to normalisation.
    \end{proof}
 \section{Estimates of the first non-trivial eigenvalue and of counting functions}\label{sec:mainresults}
 From \cite{NetrusovSafarov2005} by Netrusov and Safarov one can conclude that for bounded variation domains $\Omega$ one has $N_N(\Omega,t) = C_W ^{(n)} t^{n/2} + \mathcal{O}(t^{\delta/2})$, where $\delta := \overline{\dim_M}(\partial \Omega,\Omega)$ is the upper inner Minkowski dimension. The aim of this section is to obtain an analogue asymptotic result on $N_{N}(\Omega,t)$ for more general domains, with focus on domains with self-similar boundary, thus generalising the results and partially resolving a question raised in \cite[Sec.~5.4.2]{NetrusovSafarov2005} in Thm.~\ref{thm:wellcovered-bv} and Rem.~\ref{rem:question}. Key for obtaining such an asymptotic result for $N_N(\Omega,t)$ is the construction of a foliation in $\Omega_{-\epsilon}$. 
\subsection{Foliations}\label{subsec:foliation}
The rough procedure for obtaining asymptotic results on $N_N(\Omega,t)$ is to partition $\Omega$ into $\Omega_{-\epsilon}$ and $\Omega\setminus \Omega_{-\epsilon}$. For $\Omega\setminus \Omega_{-\epsilon}$ we will use Whitney covers, and for $\Omega_{-\epsilon}$ we will use paths that connect points which are close to the boundary $\partial\Omega$ to points located in a more controlled region away from the boundary (close to $\partial \Omega_{-\epsilon}$). If a family of such paths partitions a domain, we will call it a \emph{foliation} of the domain. If the corresponding volume form is bounded in a certain sense (see Def.~\ref{def:wellfoliated}), such a family of paths allows an estimate of the second Neumann eigenvalue of the domain. Following an idea by Brolin in \cite{brolin1965}, we develop a method to construct foliations for domains whose boundary locally admits IFS structure. 
In the below, we will restrict ourselves to the two-dimensional setting and impose some assumptions for ease of exposition. Note that generalisations to higher dimensions as well as more general settings of IFS can be formulated (see for example Rem.~\ref{rem:conformal}).\par
To state our assumptions, let us recall some notions. By an \emph{iterated function system (IFS)} we understand a family $\Phi = \{\phi_i\}_{i \in \Sigma}$ of contracting maps $\phi_i\colon\mathbb R^n\to\mathbb R^n$ with finite index set $\Sigma$. The \emph{limit set} associated to the IFS $\Phi$ is the unique non-empty compact set satisfying $M=\bigcup_{i\in\Sigma}\phi_i(M)=:\Phi(M)$. The limit set is called \emph{self-similar} if all the $\phi_i$  are \emph{similarities}, i.\,e.\ $| \phi_i(x)-\phi_i(y) | = r_i| x-y|$ for any $x,y\in\mathbb R^2$ and some $r_i\in(0,1)$. We say that $\Phi$ satisfies the \emph{open set condition (OSC)} if there exists an open and bounded $O\subset\mathbb R^2$ such that $\phi_i(O)\subset O$ and $\phi_i(O)\cap \phi_j(O)=\emptyset$ for $i\neq j\in\Sigma$.\par
The domains $\Omega \subset \mathbb{R}^2$ that we consider here shall satisfy the following.
$\partial \Omega$ has a (local) self-similar IFS structure, i.\,e.\ $\partial \Omega$ is a finite union of self-similar sets. Suppose that $M \subset \partial \Omega$ is the limit set of an IFS $\Phi = \{\phi_i\}_{i \in \Sigma}$ with $\Sigma :=\{1,\cdots,m\}$ satisfying the open set condition with all $\phi_i$ being similarities.
There exists a closed interval $I_0$ embedded in some affine plane $\mathbb{R}^{2-1}$, such that $\Phi^k I_0$ converges to $\neig$ in the Hausdorff metric as $k\to\infty$. 
By applying translation, rotation and scaling we can assume without loss of generality that $I_0=[0,1]\times \{0\}\subset\mathbb R^2$. Suppose that $\Phi(I_0)\subset \mathbb R\times [0,\infty)$ 
and that $\phi_i(I_0)\cap \phi_j(I_0)\neq\emptyset$ if and only if $j\in\{i-1,i+1\}$. In other words the maps $\phi_i$ are numbered in such a way that only neighbouring $\phi_i(I_0)$ intersect and we assume that the intersection is a singleton. 
We introduce a foliation (called \emph{seed foliation}) of the domain $G_1$ bounded by $I_1:=\Phi(I_0)$ and $I_0$ for which any fibre runs from $I_0$ to $\Phi(I_0)$. This gives rise to a map $b_1\colon I_0 \to \Phi(I_0)$ that maps $x \in I_0$ to the corresponding point in $\Phi(I_0)$ that lies in the same fibre as $x$, see Fig.~\subref{fig:KochSeed} for an example.\par
\begin{figure}[h!]
    \centering
    \begin{subfigure}[t]{0.45\textwidth}
        \centering
        \begin{tikzpicture}[
 scale=2.3,decoration=Koch snowflake]
    \draw[dashed,line width=0.01mm] decorate{ (0,0) -- (3,0) };
    \draw[fill=black!20!white!40!,line width=0.01mm,dashed] decorate{ (0,0) -- (3,0) };
    \draw[line width=0.5mm] (0,0) -- (3,0);
    \node[scale=1,above] at (0.5,0) {\small$\phi_1(I_0)$};
    \node[scale=1,above] at (1.13,0.5) {\small$\phi_2(I_0)$};
    \node[scale=1,above] at (1.87,0.5) {\small$\phi_3(I_0)$};
    \node[scale=1,above] at (2.5,0) {\small$\phi_4(I_0)$};
    \foreach \k in 
    {0,5,10,15,20,25,30,35,40}{
        \draw[gray] (1.5+\k/80,0) -- (1.5+\k/80,{-sqrt(3)*(1.5+\k/80)+2*sqrt(3)});
        \draw[gray] (1.5-\k/80,0) -- (1.5-\k/80,{-sqrt(3)*(1.5+\k/80)+2*sqrt(3)});
        }
    \foreach \m in {5,10,...,240}{
        \draw[dotted] (\m/80,0) -- (\m/80,-0.2);
    }
    \draw (0,0) -- (0,-0.2) -- (3,-0.2) -- (3,0);
    \node at (0.75,-0.1) {\small$E$};
    \node at (1.5,-0.08) {\small$I_0$};
\end{tikzpicture}
        \subcaption{Example of a seed foliation in $G_1$.
It connects $I_0$ (a unit interval) and $\Phi(I_0)$ (the first iteration of the Koch curve, dashed) given by its classic four maps $\{\phi_i\}_{i=1,\ldots,4}$ in a pre-described manner. Additionally a box $E$ has been added underneath $I_0$ as needed in Lem.~\ref{lem:ewestimate} and the fibres have been extended (in dotted lines) through $E$.}\label{fig:KochSeed}
    \end{subfigure}%
    \qquad
    \begin{subfigure}[t]{0.45\textwidth}
        \centering
        \begin{tikzpicture}[scale=0.0245]
 
    \def\x{2};
    \def\y{5};
    \draw[thick] (\x*\x,0) -- (\x*\x+2*\x*\y+\y*\y,0);
    \draw[thick] (\x*\x+4*\x*\y+4*\y*\y,0) -- (\x*\x+6*\x*\y+9*\y*\y,0);
    \draw[domain=0:\y, smooth, variable=\t] plot ({(\x+\y+\t/2)*(\x+\y+\t/2)-\t*\t*3/4)}, {sqrt(3)*\t*(\x+\y+\t/2)});
    \draw[domain=0:\y, smooth, variable=\t] plot ({(\x+2*\y-\t/2)*(\x+2*\y-\t/2)-\t*\t*3/4)}, {sqrt(3)*\t*(\x+2*\y-\t/2)});
     \draw[thick] (\x*\x,0) -- (\x*\x,-3*\y) -- (\x*\x+6*\x*\y+9*\y*\y,-3*\y) -- (\x*\x+6*\x*\y+9*\y*\y,0);
     \foreach \m in {1,9,...,280}{
         \draw[dotted,  thick] (\x*\x+\m,0) -- (\x*\x+\m,-3*\y);}
     \foreach \m in {4,12,...,40}{
         \draw[thick, domain=0:sqrt(3)*(sqrt((\x+\y)*(\x+\y)+\m)-\x-\y), smooth, variable=\t] plot ({(\x+\y)*(\x+\y)+\m-\t*\t}, {2*sqrt((\x+\y)*(\x+\y)+\m)*\t});}
     \foreach \m in {44,52,...,92}{
         \draw[thick, domain=0:sqrt(3)*(\x+2*\y-sqrt((\x+\y)*(\x+\y)+\m)), smooth, variable=\t] plot ({(\x+\y)*(\x+\y)+\m-\t*\t}, {2*sqrt((\x+\y)*(\x+\y)+\m)*\t});}
\end{tikzpicture}
        \caption{Example of a seed foliation of the conformally distorted Koch curve described in Rem.~\ref{rem:conformal} with $I_0=[2,17]$ under the conformal map $f\colon\mathbb R^2\setminus B_1(0)\to \mathbb R^2\setminus B_1(0)$, $(x,y)\mapsto(x^2-y^2,2xy)$.}\label{fig:conf-Koch-seed}
    \end{subfigure}
\end{figure}
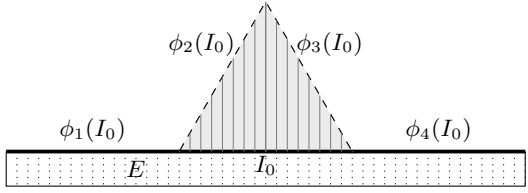
Next, we show how the IFS leads to a foliation of $\bigcup_{k \in \mathbb{N}} \overline{\Phi^k(G_1)}$. For $k\geq 1$, let $G_k$ be the domain bounded by $I_0$ and $\Phi^k(I_0)$ and let $H_k$ be the closure of the domain bounded by $\Phi^{k-1}(I_0)$ and $\Phi^{k}(I_0)$, i.\,e.\ $H_k = \overline{G_{k}\backslash G_{k-1}}$. Then $\Phi$ maps $H_k$ into $H_{k+1}$ and this mapping is surjective. Thus, any point $x \in \bigcup_k G_k$ lies in some $H_k$ and there is a fibre in $H_1$ whose image under some iteration of maps of the IFS runs through $x$. In order to describe the foliation, it is sufficient to describe the construction of such fibres through points in $\Phi^k(I_0)$. For such a point $x$ there is a finite word $w = (w_1w_2\cdots w_k) \in \Sigma^k$ with $x \in \phi_w(I_0) := \phi_{w_k}\phi_{w_{k-1}}\cdots \phi_{w_1}(I_0)$.\footnote{Note that this convention is useful here as it allows efficient use of the left-shift operator $\sigma$ instead of a right-shift.} This gives rise to a map $b_k ^{-1}:\Phi^k(I_0) \to \Phi^{k-1}(I_0)$ given by $\phi_w(I_0) \ni x \mapsto \phi_{\sigma w} b_1 ^{-1} \left(\phi_{\sigma w}\right)^{-1}(x) \in \phi_{\sigma w}(I_0)$ that maps each point $x \in \Phi^k(I_0)$ to the point at which the fibre through $x$ intersects $\Phi^{k-1}(I_0)$. The concatenation $b_k b_{k-1}\cdots b_1$ then constructs a fibre $\gamma_q$ from a unique $q \in I_0$ to $x$ showing that $\bigcup_k G_k$ is foliated.
\par
Note that each $\gamma_q$ is of finite length $\len\gamma_q$, which can be seen as follows.
Let $L$ denote the supremum over the lengths of all fibers that connect a point in $I_0$ to a point in $\Phi(I_0)$. The above construction ensures that the length of any fibre connecting a point in $I_0$ to a point in $\Phi^k(I_0)$ is bounded from above by $L\cdot\sum_{\ell=0}^{k-1}r^{\ell}\leq L/(1-r)$, where $r<1$ denotes the supremum over the contraction ratios of $\phi_1,\ldots,\phi_m$. As the upper bound is independent of $k$, it gives a uniform upper bound.
This observation implies that we can parameterise fibres $\gamma_q$ by $\gamma_q(t)$ for $t\in[0,\len\gamma_q]$. We let $\varphi:(q,t) \mapsto \gamma_q(t)$ denote the associated change of coordinates, let $q \in \text{int}(I_0)$ be fixed and $t>0$ be such that $x:=\gamma_q(t) \in \text{int}(H_{s+1})$ for some $s$. Further, we let $v_s$ denote the intersection point of $\Phi^s(I_0)$ with the fibre through $x$ and set $\widetilde{b_1}:= \varphi|_{\varphi^{-1}(G_1)}$. Analogously to the construction of $b_k$ we define $\widetilde{b_k}$ mapping $(q_{{k-1}},t)$ with $q_{k-1}\in\Phi^{k-1}(I_0)$ and $t$ small enough to the point in $H_k$ whose fibre runs through $q_{{k-1}}$ and has a length $t$ from $q_{{k-1}}$. Then $\varphi(q,t)=\widetilde{b_{s+1}}(b_s \cdots b_1 q,\tilde{t})$ where $\widetilde{t}$ is the length of the fibre through $q$ from $b_s\cdots b_1 q$ to $\varphi(q,t)$. Therefore $t-\widetilde{t}$ is constant for small enough variations of $t$. Supposing that $w \in \Sigma^s$ is such that $v_s\in\phi_{\omega}(I_0)\subset \Phi^s(I_0)$ the density function $\beta:=\lvert \det D\varphi \rvert$ of the change of coordinates given by $\varphi$ satisfies 
\begin{align*}
    \beta(q,t) &= \left| \det \left(  \prod_{\ell=1} ^{s} Db_{\ell}|_{b_{\ell-1}\cdots b_1(q)}\right) \right| \cdot \left|\det D\widetilde{b_{s+1}}|_{(b_s\cdots b_1 q,\tilde{t})}\right| \\
    &= \left| \det \left(  \prod_{\ell=1} ^{s} Db_1|_{(\phi_{w|_{\ell-1}}) ^{-1} b_{\ell-1}\cdots b_1(q)}\right) \right| \cdot \underbrace{\left|\det D\widetilde{b_{1}}|_{(\phi_w)^{-1}b_s\cdots b_1 q,\tilde{t}}\right|}_{\in \text{image}_{x \in G_1} \beta(x)},
\end{align*}
since all $\phi_i$ are similarities. This quantity is readily computed for triangles, which is the only relevant case here, see Fig.~\ref{fig:FasernDreieck}. 
The collection of fibres thus constructed gives a partition of the domain into rectified curves that are differentiable almost everywhere. We add a bounded domain $E$ with $I_0\subset\partial E$ (e.\,g.\ a rectangle) to which we can extend each fibre for some $t < 0$. Thus, every fibre will run through $E$ before passing $I_0$. 
We adjust the parametrisation so that any such extended path $\gamma$ is parameterised by some interval $(0,\len\gamma)$.
\begin{figure}[h!]
 \centering
 \begin{tikzpicture}
 \draw (0,0) -- (3,0) -- (4,1) -- (0,0);
 \foreach \k in {5,10,15,20,25,30,35,40,45,50,55}{
        \draw[gray] (3*\k/55,0) -- (3*\k/55+\k/55,\k/55);
 }
\node[below] at (1.5,0) {$a$};
\node[above] at (2,0.5) {$b$};
\node[right] at (3.5,0.4) {$c$};
\end{tikzpicture}
\caption{Basic case of a 'natural' foliation in a triangle. In such a situation an elementary calculation reveals that $|\det Db_1(q)| = \frac{b}{a}$.}\label{fig:FasernDreieck}
\end{figure}
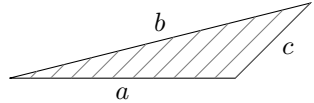
 \begin{defn}\label{def:wellfoliated}
  A domain $D$ is called \emph{well-foliated} if there exists a collection $\{\gamma\}_{\gamma \in \Gamma}$ of paths in $D$ as above with $0 < \inf_{\gamma \in \Gamma} \len(\gamma) \leq \sup_{\gamma \in \Gamma} \len(\gamma) < L$ such that
  the density $\beta$ is bounded in the sense that
  \begin{align} 
   &\essinf_{x \in D} \beta(x) > 0\ \text{and}\label{eq:betacond1}\\
   &\sup_{\gamma \in \Gamma} \int_0 ^{\len(\gamma)} \beta(\gamma,t)dt < \infty.\label{eq:betacond2}
  \end{align}
  In this setting, let $r \in (0,\inf_{\gamma \in \Gamma} \len\gamma)$ be fixed and define $E:=\{\gamma(t):\gamma \in \Gamma,0 < t < r\}$. For $S \subset D$, we write $\beta_{\inf}(S):=\essinf_{x \in S} \beta(x)$ and $\mathcal{I}_\beta(S):=\sup_{\gamma \in \Gamma} \int_0 ^{\len \gamma} 1_S(\gamma(t))\beta(\gamma(t))dt$.
 \end{defn}
 Note that we can identify a path $\gamma \in \Gamma$ with its intersection point $\gamma(0) = q \in I_0$. 
 \begin{lem}\label{lem:ewestimate}
  Let $D \subset \mathbb{R}^n$ be a domain that is well-foliated by $\{\gamma\}_{\gamma \in \Gamma}$ as in Def.~\ref{def:wellfoliated}. Then for all $\alpha>0$
  \begin{align*}
   \lambda_2 ^N(D) &\geq \left(\lambda_2 ^N (E)^{-1} \left(1  + \frac{1+\alpha}{r}\frac{ \mathcal{I}_\beta {(D\backslash E)} }{\beta_{\inf}{(E)} }\right) + \left(1+\alpha^{-1}\right) L \frac{\mathcal{I}_\beta (D\backslash E) }{\beta_{\inf}(D)} \right)^{-1}.
  \end{align*}
 \end{lem}
 \begin{proof}
 Let $u \in H^1(D) \cap 1^\perp$ and set $\overline{u_E} := \frac{1}{\vol E}\int_E udx$. Then $\|u \|^2 _{L^2(D)} \leq \|u - \overline{u_E}\|^2 _{L^2(D)}$. In particular, by Cor.~\ref{cor:minmax2}, it suffices to show the following Poincar\'{e}-Wirtinger inequality
  \begin{align}
  \|u-\overline{u_E}\|_{L^2(D)} ^2 \leq C\|\nabla u\|_{L^2(D)} ^2  , \label{eq:1n}
 \end{align}
 with $C^{-1}$ given by the right hand side of the claim. To this end, we subdivide the expression
 \begin{align*}
  \|u-\overline{u_E}\|_{L^2(D)} ^2 = \underbrace{\int_{E} |u(y) - \overline{u_E}|^2 dy}_{:=I_1} + \underbrace{\int_{D\backslash E} |u(x) - \overline{u_E}|^2 dx}_{:=I_2},
 \end{align*}
 and estimate $I_1$ and $I_2$ separately. By Cor.~\ref{cor:minmax2}, one has
 \begin{align}
  I_1 \leq \lambda^N _2(E) ^{-1}\|\nabla u\|_{L^2(E)} ^2 \leq \lambda^N _2 (E) ^{-1} \|\nabla u\|_{L^2(D)} ^2. \label{eq:esti1}
 \end{align}
 Recall that for any $a,b,c \in \mathbb{C}$ and $\alpha>0$, there is the following inequality: $|a-b|^2 \leq (1+\alpha)|a-c|^2 + (1+\alpha^{-1})|b-c|^2$. We now consider
 \begin{align*}
  I_2 &= \int_{D \backslash E} |u(x) - \overline{u_E}|^2dx = \int_\Gamma \int_{r} ^{\len \gamma} |u(\gamma(t)) - \overline{u_E}|^2 \beta(\gamma,t) dt d\gamma,\\
  &= \frac{1}{r} \int_\Gamma \int_{r} ^{\len \gamma} \int_0 ^{r} |u(\gamma(t)) - \overline{u_E}|^2\beta(\gamma,t) dt' dt d\gamma\\
  &\leq \underbrace{\frac{1+\alpha}{r} \int_\Gamma \int_{r} ^{\len \gamma} \int_0 ^{r} |u(\gamma(t')) - \overline{u_E}|^2\beta(\gamma,t) dt' dt d\gamma}_{=:I_2 '}\\
  &\qquad\qquad+ \underbrace{\frac{1+\alpha^{-1}}{r}\int_\Gamma \int_{r} ^{\len \gamma} \int_0 ^{r} |u(\gamma(t)) - u(\gamma(t'))|^2\beta(\gamma,t) dt' dt d\gamma}_{=:I_2 ''},
 \end{align*}
 which holds for any $\alpha >0$.  By assumption, $\int_{r} ^{\len \gamma} \beta(\gamma,t)dt \leq \mathcal{I}_\beta(D\backslash E)$ for all $\gamma \in \Gamma$. Hence
 \begin{align*}
  I_2 ' &= \frac{1+\alpha}{r} \int_\Gamma \int_{r} ^{\len \gamma} \int_0 ^{r} |u(\gamma(t')) - \overline{u_E}|^2\beta(\gamma,t) dt' dt d\gamma\\
  &\leq \frac{1+\alpha}{r}\mathcal{I}_\beta {(D\backslash E) } \int_\Gamma \int_{0} ^{r} |u(\gamma(t')) - \overline{u_E}|^2 dt' d\gamma \leq \frac{1+\alpha}{r}\frac{\mathcal{I}_\beta {(D\backslash E)} }{\beta_{\inf} {(E)}} \int_\Gamma \int_{0} ^{r} |u(\gamma(t')) - \overline{u_E}|^2 \beta(\gamma,t') dt' d\gamma\\
  &= \frac{1+\alpha}{r}\frac{\mathcal{I}_\beta {(D\backslash E)} }{\beta_{\inf} {(E)}} \|u - \overline{u_E} \|^2 _{L^2(E)} \leq \lambda_2 ^N(E) ^{-1} \frac{1+\alpha}{r}\frac{\mathcal{I}_\beta {(D\backslash E)} }{\beta_{\inf} {(E)}} \|\nabla u\|^2 _{L^2(D)},
 \end{align*}
 and
 \begin{align*}
  I_2 '' &= \frac{1+\alpha^{-1}}{r}\int_\Gamma \int_{r} ^{\len \gamma} \int_0 ^{r} |u(\gamma(t))-u(\gamma(t'))|^2\beta(\gamma,t) dt' dt d\gamma\\
  &=\frac{1+\alpha^{-1}}{r}\int_\Gamma \int_{r} ^{\len \gamma} \int_0 ^{r} \left| \int_{t'} ^t \partial_s u(\gamma(s))|_{s=\sigma}d\sigma \right|^2\beta(\gamma,t) dt' dt d\gamma \\
  &\stackrel{(\dagger)}{\leq} \frac{1+\alpha^{-1}}{r}\int_\Gamma \int_{r} ^{\len \gamma} \int_0 ^{r} (t-t') \int_{t'} ^t \left|\partial_s u(\gamma(s))|_{s=\sigma}\right|^2d\sigma \beta(\gamma,t) dt' dt d\gamma \\
  &\leq \frac{1+\alpha^{-1}}{r}\int_\Gamma \int_{r} ^{\len \gamma} \int_0 ^{r} (t-t') \int_{t'} ^t \left|\nabla u(x)|_{x=\gamma(\sigma)}\right|^2d\sigma \beta(\gamma,t) dt' dt d\gamma \\
  &\leq \frac{1+\alpha^{-1}}{r}\int_\Gamma \int_{r} ^{\len \gamma} \int_0 ^{r} L \int_{0} ^{\len \gamma} \left|\nabla u(x)|_{x=\gamma(\sigma)}\right|^2d\sigma \beta(\gamma,t) dt' dt d\gamma \,,\qquad \text{ as }t-t'\leq L\\
  &\leq (1+\alpha^{-1})L \mathcal{I}_\beta (D\backslash E) \int_\Gamma \int_{0} ^{\len \gamma} \left|\nabla u(x)|_{x=\gamma(\sigma)}\right|^2d\sigma d\gamma \\
  &\leq (1+\alpha^{-1})L \frac{\mathcal{I}_\beta (D\backslash E) }{\beta_{\inf}{(D)}} \int_\Gamma \int_{0} ^{\len \gamma} \left|\nabla u(x)|_{x=\gamma(\sigma)}\right|^2 \beta(\gamma,\sigma) d\sigma d\gamma = (1+\alpha^{-1})L \frac{\mathcal{I}_\beta (D\backslash E) }{\beta_{\inf}{(D)}} \| \nabla u \|^2 _{L^2(D)}.
 \end{align*}
 where we used an application of Jensen's Theorem in $(\dagger)$.\footnote{Using Jensen and $|\nabla \gamma_q(s)|^2 =1$, $|s|\leq 1$ and $|\partial_s u| = |\langle s,\nabla u\rangle| \leq |\nabla u|$,
 \begin{align*}
  |u(x)-u(y)|^2 &= \left|\int_{t_0} ^t  \partial_s u(\gamma_q(s))  ds\right|^2 \leq (t-t_0) \int_{t_0} ^t \left| \partial_s u(\gamma_q(s)) \right|^2 ds\\
  &\leq(t-t_0) \int_{t_0} ^t \left| \nabla u \right|^2 ds \leq (t-t_0) \int_0 ^{\len \gamma_q} \left|\nabla u|_{\gamma_q(s)}\right|^2 ds.
 \end{align*}}
 \end{proof}
 \begin{rem} 
 \begin{enumerate} \label{rem:conformal}
    \item Lem.~\ref{lem:ewestimate} generalises results of \cite{NetrusovSafarov2005}. A comparison with the results of \cite{NetrusovSafarov2005} will be presented in Sec.~\ref{subsec:bdvar}.
    \item Note that Lem.~\ref{lem:ewestimate} and all results below hold for the higher dimensional setting, where $\Omega\subset\mathbb R^n$ with arbitrary $n\in\mathbb N$. Merely the construction of a foliation is described only for the case $n=2$, for ease of exposition of the construction and formulation of suitable conditions.  
    \item The above construction is presented for IFS consisting of similarities for ease of exposition. However, the construction extends to more general IFS. For example, a seed foliation can be constructed for the image $f(V)$ of the Koch curve $V$ under a conformal map $f$ as shown in Fig.~\subref{fig:conf-Koch-seed}. Notice that $f(V)$ is a self-conformal set and that the corresponding conformal IFS can be used to transport the seed foliation to obtain a foliation of the region bounded between the $x$-axis and the conformal Koch curve. 
    In Fig.~\subref{fig:conf-Koch-seed}, we applied $f$ to the seed foliation of $V$ shown in Fig.~\subref{fig:KochSeed}, but note that other constructions for seed foliations are possible in more general settings.
    \end{enumerate}
 \end{rem}
 \begin{defn}\label{def:wellcovered}
  A domain $\Omega \subset \mathbb{R}^n$ is called \emph{well-covered} if there exists $\mu \in \mathbb{N}$ and $\epsilon_0 >0$ such that for any $\epsilon\in(0,\epsilon_0]$ there is a cover of $\Omega_{-\epsilon}$ by well-foliated domains $\{D_i^{\epsilon}\subset\Omega\}_{i\in I_{\epsilon}}$ satisfying $\vol_n(\Omega_{-\epsilon})=\vol_n(\Omega_{-\epsilon}\cap (\bigcup_i D_i^{\epsilon}))$ with corresponding quantities $r,L,\mathcal{I}_\beta$ (as in Def.~\ref{def:wellfoliated}) and uniformly bounded multiplicity $\leq \mu$ such that
  \begin{enumerate}
   \item \label{item:zweidef:1} There are $\delta \in \mathbb{R}$ and $C(\Omega) \in \mathbb{R}$ such that $\# I_\epsilon \leq C(\Omega) \epsilon^{-\delta}$ for all $\epsilon\in(0,\epsilon_0]$.\label{item:dreiindef}
   \item There exist constants $c_r ^\pm, c_L ^\pm, c_{\mathcal{I}} ^\pm , c_{\diam} ^\pm, c_{\vol} ^\pm >0$ such that for all $\epsilon\in(0,\epsilon_0]$ and for each $D_i ^\epsilon$ one has 
   \label{item:zweidef}
   \begin{equation*}
    \tfrac{r}{\epsilon} \in (c_r ^-, c_r ^+),\
    \tfrac{L}{\epsilon} \in (c_L ^-, c_L ^+),\
    \tfrac{\mathcal{I}_\beta}{\epsilon} \in (c_{\mathcal{I}} ^-, c_{\mathcal{I}} ^+),\
    \tfrac{\diam D_i ^\epsilon}{\epsilon} \in (c_{\diam} ^-, c_{\diam} ^+),\
    \tfrac{\vol_n(D_i ^\epsilon)}{\epsilon^n} \in (c_{\vol} ^-,c_{\vol} ^+).
   \end{equation*}
   \item \label{item:zweidef:3}$\inf_{\epsilon,i} \beta_{\inf}(D_i^{\e})>0$.
   \item \label{item:zweidef:4}For any $D_i ^\epsilon$ the corresponding $E$ must have first non-trivial Neumann eigenvalue $\lambda_2 ^N(E) \geq c_E r ^{-2}$ for a fixed $c_E$ independent of $\epsilon$.\label{item:NeumannEWvonB}
  \end{enumerate}
 \end{defn}
   To simplify notation, we sometimes use $\ll$, $\gg$ and $\asymp$ with the following meaning: For two real-valued functions $f,g$ with common domain of definition $D$ we write $f \ll g$ if there exists a constant $c\in\mathbb R$ such that $f(x) \leq cg(x)$ holds for all $x\in D$. Moreover, we write $f \asymp g$ if $f \ll g$ and $f \gg g$.
 \begin{rem}
    \begin{enumerate}
  \item Condition \ref{item:dreiindef} in Def.~\ref{def:wellcovered} is closely related to upper inner Minkowski measurabilty. To see this, notice that Def.~\ref{def:wellcovered}\ref{item:dreiindef} implies
  \begin{align*}
   \vol_n(\Omega_{-\epsilon}) \leq \sum_{i \in I_\epsilon} \vol_n D_i ^\epsilon \leq c_{\vol} ^+ \epsilon^n \# I_\epsilon \leq c_{\vol} ^+ C(\Omega) \epsilon^{n-\delta}.
  \end{align*}
  If $\delta$ coincides with the upper inner Minkowski dimension of $\partial \Omega$, the above equation implies that $\overline{\mathcal{M}}_\delta(\partial \Omega,\Omega)<\infty$. Moreover, if $\Omega$ is upper inner Minkowski measurable, then by definition of the inner upper Minkowski content, there is an error $\epsilon'(\epsilon)$ with $\limsup_\epsilon \epsilon'(\epsilon) = 0$ (cf. proof of Prop.~\ref{prop:whitneycardinality}) such that
  \begin{align*}
  \#I_\epsilon &\leq \frac{1}{\epsilon^n c_{\vol} ^-} \sum_{i \in I_\epsilon} \vol_n D_i ^\epsilon  
  \leq \frac{\mu}{c_{\vol} ^-} \vol_n \Omega_{-\epsilon {(c_{\diam} ^++1)}} \epsilon^{-n} 
  \\ 
  &\leq \frac{\mu}{c_{\vol} ^-} \overline{\mathcal{M}}_\delta(\partial \Omega,\Omega) \epsilon^{-\delta}(c_{\diam}^+ +1)^{n-\delta}\left(1+ \epsilon'(\epsilon (c_{\diam}^+ +1))\right),
  \end{align*}
  Notice that through $\epsilon'$, yet another geometric quantity of $\partial \Omega$, namely the speed of convergence of $\vol_n (\Omega_{-\epsilon}) \epsilon^{\delta -n}$ enters the final estimate.
  \item The reverse condition in Def.~\ref{def:wellcovered}.\ref{item:NeumannEWvonB} is automatically satisfied for all domains to which \eqref{eq:rayleigh} applies. More precisely, based on a result by Szeg\H{o} in \cite{szego1954}, Weinberger showed in \cite{weinberger1956} a Faber-Krahn-type isoperimetric inequality:
  \begin{align*}
   \lambda_{2} ^N(\Omega) \leq \lambda_{2} ^N(\Omega^\ast) = \frac{p_n ^2}{\left(\frac{1}{2}\diam \Omega^\ast\right) ^2} = p_n ^2 \left( \frac{\omega_n}{\vol_n \Omega} \right)^{\frac{2}{n}},
  \end{align*}
  where $\Omega^\ast$ is the ball with $\vol_n(\Omega^\ast)=\vol_n (\Omega)$, $\omega_n$ is the volume of the $n$-dimensional unit ball and $p_n$ is the first positive zero of $\left( x^{1-\frac{n}{2}}J_{\frac{n}{2}}(x) \right)'$, i.\,e.\ the smallest positive solution of $J_{\frac{n}{2}}(x) =xJ_{1+\frac{n}{2}}(x)$ where $J_m(x)$ is the $m$-th spherical Bessel function. Therefore one may replace the condition in Def.~\ref{def:wellcovered}.\ref{item:NeumannEWvonB} with $\lambda_2 ^N(E) \asymp \epsilon^{-2}$. 
  \item Bi-Lipschitz maps preserve the property of being well-covered: Let $f:\Omega \to \Omega'$ be a bi-Lipschitz map and let $\Omega$ be well-covered. Then $\Omega'$ is also well-covered. Indeed, if a family of covers $\{D_i ^\epsilon\}_{i \in I_\epsilon}$ satisfies the conditions of Def.~\ref{def:wellcovered}, then $\{f(D_i ^\epsilon)\}_{i \in I_\epsilon}$ can be used to cover $(f\Omega)_{-\epsilon'}$ with $\epsilon \asymp \epsilon'$.
  \end{enumerate}
  \end{rem}
 \begin{cor}[\textquotedblleft{}$\lambda_2 ^N(D) \asymp \diam(D)^{-2}$\textquotedblright{}]\label{cor:ewestimate}
 Let $\Omega$ be well-covered and let $\{D^\epsilon _i \}_{i \in I_\epsilon}$ be a cover of $\Omega_{-\epsilon}$ consisting of well-foliated domains as in Def.~\ref{def:wellcovered}. Then by Lem.~\ref{lem:ewestimate},
 \begin{align*}
  \lambda_2 ^N(D^\epsilon _i) \geq C_1(\Omega_{-\epsilon})\epsilon^{-2}
 \end{align*}
\end{cor}
  \subsection{Sets of bounded variation as well-foliated domains}\label{subsec:bdvar}
  In this section we show that the class of well-covered domains contains the class of domains of bounded variation. For the class of domains of bounded variation bounds on Laplace eigenvalue counting functions have been found in \cite{NetrusovSafarov2005}.
  \begin{defn}[Def.~1.1-1.2 in \cite{NetrusovSafarov2005}]
   Let $\Omega' \subset \mathbb{R}^{n-1}$ be a domain and let $f \in C^0(\overline{\Omega'},\mathbb{R})$. We define the \emph{oscillation of $f$ on $\Omega'$} as
   \begin{align*}
    \Osc(f,\Omega') := \frac{1}{2}\left( \sup_{x \in \Omega'} f(x) - \inf_{x \in \Omega'} f(x) \right).
   \end{align*}
  Let $Q_n$ be an $n$-dimensional cube with arbitrary size whose edges are parallel to the axes.
  \begin{enumerate}
   \item Let $f:Q_n \to \mathbb{R}$ be bounded. Then for any $\delta > 0$, we define $\mathcal{V}_\delta(f,Q_n)$ as the maximal number of disjoint open cubes $Q_{n,i} \subset Q_n$ whose edges are parallel to coordinate axes with $\Osc(f,Q_{n,i}) \geq \delta$ for all $i$. If $\Osc(f,Q_n)< \delta$, we set $\mathcal{V}_\delta(f,Q_n):=1$.
   \item Let $\tau: \mathbb{R}_{>0} \to \mathbb{R}_{>0}$ be non-decreasing. The space spanned by the continuous functions $f:\overline{Q_n} \to \mathbb{R}$ defined on some cube $Q_n$ such that $\mathcal{V}_{1/t}(f,Q_n) \leq \tau(t)$ for all $t$, is called the set of functions with $(\tau,\infty)$-bounded variation and denoted by $\BV_{\tau,\infty}(Q_n)$.
  \end{enumerate}
  For a domain $\Omega \subset \mathbb{R}^n$, we write $\Omega \in \BV_{\tau,\infty}$ iff for any point $p \in \partial \Omega$ there is an open neighbourhood $U_p \subset \mathbb{R}^n$ such that (up to an $SO(n)$), $U_p \cap \Omega$ is the set of points below the graph of some $f \in BV_{\tau,\infty}(Q_{n-1})$ and a constant function $b< \inf f$. 
  \end{defn}
  \begin{thm}
   Let $\tau(t) \asymp t^\delta$ for some $\delta \geq 1$. Then the set of well-covered domains contains $\BV_{\tau,\infty}$ as a proper subset.
   \label{thm:wellcovered-bv}
  \end{thm}
  \begin{proof}
   Let $\Omega \in \BV_{\tau,\infty}$ be a domain in $\mathbb{R}^n$ and let $p \in \partial \Omega$ with corresponding neighbourhood $U_p$. Then $U_p \cap \Omega$ is bounded by a graph $\Gamma_f$ of a function $f:\overline{Q_{n-1}} \to \mathbb{R}$ from above and a constant $b < \inf f$ from below. Hence there is a foliation $\{\gamma \in \Gamma\}$ by straight lines: For any $x \in Q_{n-1}$, let $\gamma_x:t \mapsto (x_1,x_2,\cdots,x_{n-1},b)+t(0,\cdots,0,1)$ be so that $\len \gamma_x = f(x)-b>0$. Such foliations satisfy conditions \eqref{eq:betacond1}-\eqref{eq:betacond2} trivially, as the corresponding density $\beta$ is $1$ everywhere. Then for some $r \in (b,\inf f)$, we set $E:= \{x \in U_p: \exists \gamma \text{ s.\,t.\ }x = \gamma(r)\}$ showing that $U_p$ is well-foliated for all $p \in \partial \Omega$.\par   
   Since $\partial \Omega$ is compact (being closed and bounded), there is a finite set of open neighbourhoods $U_p$ corresponding to some $f \in \BV_{\tau,\infty}$ that cover $\partial \Omega$. By \cite[Thm.~3.5, Cor.~3.8]{NetrusovSafarov2005}, for any $\epsilon>0$ small enough, each of these can be covered by a family of well-foliated domains  of size $\asymp \epsilon$ with $r=\epsilon$ and of cadinality $\ll \mathcal{V}_{\epsilon/2}(f,Q_{n-1}) + \epsilon^{-n}\vol_n(\Omega_{-\epsilon}) + \epsilon^{-1}\int_1 ^{1/\epsilon} t^{-2}\tau(t)dt \ll \tau(1/\epsilon) = \epsilon^{-\delta}$.
   \par
   Since the Koch snowflake $K$ is not locally a graph, $K \notin \BV_{\tau,\infty}$ while $K$ is well-covered as shown in Sec.~\ref{subsec:application}.
  \end{proof}
  \begin{rem}\label{rem:question}
   Suppose there is a bi-Lipschitz function $f:\Omega \to \Omega'$ with Lipschitz constant $L$, such that $\Omega' \in \BV_{\tau,\infty}$ with $\tau(t) \sim t^{-\delta}$ for some $\delta>1$. Then $\Omega$ is well-covered. Such domains are studied in \cite{NetrusovSafarov2005} but in many practical cases, it appears to be easier to show a domain is well-covered instead of showing it is $\BV_{\tau,\infty}$ up to a bi-Lipschitz transformation.
   Moreover, Thm.~\ref{thm:wellcovered-bv} shows that Thm.~\ref{thm:satz1} covers a larger class than $\BV_{\tau,\infty}$ addressing the open problem raised in \cite[Sec.~5.4.2]{NetrusovSafarov2005}.
  \end{rem}
 \subsection{Bounds on eigenvalue counting functions}\label{subsec:bounds}
  \begin{thm}\label{thm:satz1}
  Let $\Omega \subset \mathbb{R}^n$ be a domain. Suppose $\Omega$ is well-covered and that the upper inner Minkowski dimension $\delta$ of $\Omega$ satisfies $\delta\in [n-1,n)$. Then there exists $M_\Omega$ given in Sec.~\ref{sec:constants} such that 
   \begin{align*}
    N_N(\Omega,t) \leq C_W ^{(n)} \vol_n(\Omega)t^{n/2} + M_\Omega t^{\delta/2},
   \end{align*}
   for all sufficiently large $t$.
  \end{thm}
  \begin{proof} 
    Let $\{Q\}_{Q \in \mathcal{W}}$ denote a Whitney cover of $\Omega$ (see Sec.~\ref{subsec:whitney}). For $\epsilon>0$ we restrict $\{Q\}_{Q \in \mathcal{W}}$ to the smallest subset $\mathcal{W}_\epsilon \subset \mathcal{W}$ of cubes volume covering $\Omega\backslash\Omega_{-\epsilon}$, i.\,e.\ those cubes that are sufficiently far away from $\partial \Omega$ and hence are sufficiently large. Suppose that $\epsilon$ is small enough so that $\Omega_{-\epsilon}$ is well-covered by well-foliated domains $\{D^\epsilon _i\}_{i\in I_\epsilon}$ with $L_i ^\epsilon \sim \epsilon$. Let $J_{\epsilon}\subset I_{\epsilon}$ be such that $D_i^{\epsilon}\cap \Omega\backslash \bigcup_{\mathcal{W}_\epsilon} Q \neq\emptyset$ for $i\in J_{\epsilon}$ and set 
    $U_\epsilon := \{D^\epsilon _i\}_{i\in J_\epsilon}$. Then $\Omega = \bigcup_{i \in J_\epsilon} D^\epsilon _i \cup \bigcup_{\mathcal{W}_\epsilon} Q$ and this volume cover can be split into two disjoint volume subcovers. Let $T_\epsilon := \{Q \in \mathcal{W}_\epsilon:  Q \cap \bigcup_{i \in J_\epsilon} D_i ^\epsilon \neq \emptyset\}$ and $U_\epsilon ^\mathcal{W}:=U_\epsilon \cup T_\epsilon$ be the set of all well-foliated domains $D_i ^\epsilon$ together with all cubes in $\mathcal{W}_\epsilon$ that interset some $D_i ^\epsilon$. Next, we restrict the Whitney cover further to the minimal set of cubes necessary to volume cover $\Omega$ given the volume cover $\{U\}_{U \in U_\epsilon ^\mathcal{W}}$ by defining $\widetilde{\mathcal{W}}_\epsilon:=\{Q \in \mathcal{W}_\epsilon: Q \cap U = \emptyset\,\,\forall U \in U_\epsilon ^\mathcal{W}\}$. These volume covers are disjoint and have multiplicity $\mu( U_\epsilon ^\mathcal{W} )\leq \mu+1$ and $\mu(\widetilde{\mathcal{W}}_\epsilon)=1$ so that $\vol_n\Omega = \vol_n \left[ \left( \bigcup_{ U_\epsilon ^\mathcal{W} }U \right) \sqcup \left(\inn\bigcup_{\widetilde{\mathcal{W}}_\epsilon} \overline{Q}\right)\right]$. Therefore
 \begin{align*}
  N_N(\Omega,t) \leq \underbrace{N_N \bigg( \bigcup_{ U_\epsilon ^\mathcal{W} } U , t \bigg)}_{:=S^\epsilon _1 (t)} + \underbrace{N_N\bigg( \inn\bigcup_{\widetilde{\mathcal{W}}_\epsilon} \overline{Q},t \bigg)}_{:=S^\epsilon _2(t)}
 \end{align*}
 and we estimate both terms separately.
 \begin{enumerate}
  \item[$S_1 ^\epsilon(t)$.] By construction one can estimate the first non-trivial eigenvalue of any Whitney cube $Q \in U_\epsilon ^\mathcal{W}$ as,
 \begin{align*}
  \lambda_2 ^N(Q) = \left(\frac{\sqrt{n}\pi}{\diam Q}\right)^2 \geq \left(\frac{\sqrt{n}\pi}{\dist(Q,\partial \Omega)}\right)^2 \geq \left( \frac{\sqrt{n}\pi}{c_{\diam} ^+ \epsilon} \right)^2.
 \end{align*} 
For large enough $\lambda$ we can choose $\epsilon$ s.\,t. $(\mu+1)\lambda = C_2 \epsilon^{-2}$. We define the auxiliary constant $1 \leq s_{\sup} := \sup_{\epsilon} \frac{1}{\epsilon}\inf \left( \nu: \bigcup_{i \in I_\epsilon} D_i ^\epsilon \subset \Omega_{-\nu} \right) \leq c_{\diam} ^+$. Any $Q \in T_\epsilon$ satisfies $\dist(Q,\partial \Omega) \geq \epsilon - \diam(Q)$ (i.\,e.\ $\diam Q \geq \frac{\epsilon}{5}$) and $\diam Q \leq \dist(Q,\partial \Omega) \leq s_{\sup}\epsilon$. Then, by Prop.~\ref{prop:whitneycardinality},
 \begin{align*}
  S_1 ^\epsilon(\lambda) &\leq \sum_{U \in U_\epsilon ^\mathcal{W}} N_N\left( U,(\mu+1) \lambda \right) \leq \#I_\epsilon + \sum_{k_- \leq k \leq k_+} \#\mathcal{W}_k \leq \# I_\epsilon + 2^{k_-\delta}\sum_{k=0} ^{\lfloor 2+\log_2(5s_{\sup}) \rfloor } {\mathfrak{M}_{\Omega}}2^{k\delta} \\
  &\leq C(\Omega)\epsilon^{-\delta} + {\mathfrak{M}_{\Omega}}\left( \frac{\sqrt{n}}{s_{\sup}}\right)^\delta \frac{(8\cdot 5 s_{\sup})^\delta -1}{2^\delta-1}\epsilon^{-\delta} \leq C_3(\Omega) \epsilon^{-\delta} = \frac{C_3(\Omega)(\mu+1)^{\delta/2}}{C_2(\Omega)^{\delta/2}} \lambda^{\delta/2} ,
 \end{align*}
 where $k_-:=\left\lfloor \log_2 \frac{\sqrt{n}}{s_{\sup}\epsilon} \right\rfloor$ and $k_+ := \left\lceil \log_2 \frac{5\sqrt{n}}{\epsilon} \right\rceil$ and hence $k_+-k_- \leq 2+\log_2 (5s_{\sup})$.
 \item[$S_2 ^\epsilon(t)$.]
 Notice that $\inn\bigcup_{Q \in \widetilde{\mathcal{W}}_\epsilon} \overline{Q}$ is a polygon whose volume is bounded by $\vol_n \Omega$. By Prop.~\ref{prop:whitneyumfang}, Prop.~\ref{prop:whitneybillard} and a result on two-term asymptotics for polygons (see Sec.~4.5 and Thm.~7.4.11 in \cite{ivrii2019I} or \cite{vassiliev86}),
 \begin{align*}
     \limsup_{t \to \infty} \frac{S_2 ^\epsilon(t)}{C_W ^{(n)} \vol_n (\Omega) t^{n/2} + \frac{C_W ^{(n-1)}}{4}A_\Omega \epsilon^{n-1-\delta} t^{(n-1)/2}} \leq 1.
 \end{align*}
 \end{enumerate}
 The estimates for $S_1 ^\epsilon(t)$ and $S_2 ^\epsilon(t)$ together show that for large enough $t$
 \begin{equation*}
  N_N(\Omega,t) \leq C_W ^{(n)} \vol_n (\Omega)t^{n/2} + M_{\Omega}t^{\delta/2}.\qedhere
 \end{equation*}
  \end{proof}
  \begin{rem}\label{rem:countingfunctions}\ 
  \begin{enumerate}
     \item The minimal value of $t$ for which the estimate of Thm.~\ref{thm:satz1} holds true depends entirely on the maximal value of $\epsilon$ in Def.~\ref{def:wellcovered}.\ref{item:dreiindef}. More precisely, if $\Omega$ is well-covered for all $\epsilon < \epsilon_0$, an upper bound for $N_N(\Omega,t)$ for all $t > t_0 :=\frac{C_2 \epsilon_0 ^{-2}}{\mu+1}$ is presented below. This variant of the above estimate involves estimating $S_2 ^\epsilon(t) \leq \sum_{Q \in \widetilde{\mathcal{W}}_\epsilon} N_N(Q,t)$.
     More precisely, setting $k' _- := \lfloor -\log_2 \frac{\diam \Omega}{\sqrt{n}} \rfloor$, one has
     \begin{align*}
         S_2 ^\epsilon(t) &\leq \sum_{Q \in \widetilde{\mathcal{W}}_\epsilon} N_N(Q,t) \leq \sum_{k = k' _-} ^{k _+} \#\mathcal{W}_{k} N_N(2^{-k}I^n,t)
     \end{align*}
     While more accurate estimates are possible, this is sufficient to provide an upper bound in terms of counting functions of cubes.
     Based on a simple lattice counting argument and a result by P\'olya in \cite{polya1960}, the Neumann counting function of a unit cube $I^n = (0,1)^n \subset \mathbb{R}^n$ has an upper bound of $N_N(I^n,t) \leq N_D(I^n, (\sqrt{t}+ 2\pi\sqrt{n})^2) \leq C_W ^{(n)} (\sqrt{t}+ 2\pi\sqrt{n})^n$.\footnote{Proof: $N_N(I^n,t) = \#\{k=(k_1,\cdots,k_n) \in \mathbb{N}_0 ^n :|k|_2 ^2 \leq t/\pi^2\}$ and $N_D(I^n,t) = \#\{k=(k_1,\cdots,k_n) \in \mathbb{N} ^n:|k|_2 ^2 \leq t/\pi^2\}$. Defining $\widetilde{N_N}(r):=\#\{k \in \mathbb{N}_0 ^n: |k|_2 ^2 \leq r^2\}$ one observes $N_N(I^n,t) = \widetilde{N_N}(\sqrt{t}/\pi)$ and analogously for $N_D$. Notice that $\widetilde{N_N}(r) \leq \widetilde{N_D}(r + 2\sqrt{n})$ since one increases the radius sufficiently. Thus $N_N(I^n,t) = \widetilde{N_N}(\sqrt{t}/\pi) \leq \widetilde{N_D}(\sqrt{t}/\pi + 2\sqrt{n}) = N_D(I^n,(\pi(\sqrt{t}/\pi + 2\sqrt{n})^2)$. Finally $N_D(I^n,t) \leq C_W ^{(n)} t^{n/2}$ by P\'olya's result.} Applying this to the cubes in $\mathcal{W}_\epsilon$ yields an explicit upper bound of $S_2 ^\epsilon(t)$ that is not only asymptotic but holds for all $t > t_0$.
     \begin{align*}
         S_2 ^\epsilon (t) &\leq C_W ^{(n)} \sum_{k=k_- '} ^{k_+} \#\mathcal{W}_{k} \left( 2^{-k}\sqrt{t}+2\pi\sqrt{n} \right)^n \\
         &\leq C_W ^{(n)} \left( \vol(\Omega)t^{n/2} + \mathfrak{M}_\Omega \sum_{k=k_- '} ^{k_+} 2^{k\delta} \left[\left( 2^{-k}\sqrt{t}+2\pi\sqrt{n} \right)^n - \left(2^{-k}\sqrt{t} \right)^n \right] \right).
     \end{align*}
     With $S_2 ^{\epsilon} (t) \leq S_2 ^{\epsilon} (t_0)$ if $t \leq t_0$ this extends to a global estimate of $N_N(\Omega,t)$ and the coefficient of such an error term estimate will be denoted by $M_K ^{\text{abs}}$. 
     Since the estimate of $S_1 ^\epsilon$ is also correct for all $t$, this gives an absolute estimate of the remainder term of the Neumann counting function. In the regular case of $\delta = n-1$, this estimate then only provides a weaker estimate of $t^{(n-1)/2}\log t$ instead of the boundary term $t^{(n-1)/2}$. Similarly such an asymptotic term is also obtained if one uses the regular asymptotic law of the counting functions $N_N(2^{-k}I^n,t)$ as shown for example by Lapidus in \cite{la1991} 
     whenever the upper Minkowski content agrees the upper inner Minkowski content. 
     \label{item:allelambda}
      \item Let $\Omega \subset \mathbb{R}^n$ be any open set with finite volume such that $\vol_n(\Omega_{-\epsilon}) \leq \widetilde{C} \epsilon^{n-\delta}$ for some $\delta \in (n-1,n)$ and all sufficiently small $\epsilon>0$. Then, using Whitney covers and an estimate of the form of Prop.~\ref{prop:whitneycardinality}, one can find lower bounds for $N_D(\Omega,t)$ as was shown by van den Berg and Lianantonakis in \cite{vandenBerg2001}. Indeed, for any such domain $\Omega$ one has
   \begin{align*}
    N_D(\Omega,t) \geq 
    \begin{cases}
        C_W ^{(n)} \vol_n(\Omega) t^{n/2} - \frac{5\widetilde{C}}{(n-\delta)(\delta+1-n)}t^{\delta/2} &\colon\delta > n-1,\ t>0\\
        C_W ^{(n)} \vol_n(\Omega) t^{n/2} - 3\widetilde{C}t^{(n-1)/2}\left( \log \left((2\vol_n\Omega)^{2/n} t \right) \right) &\colon \delta=n-1,\ t>\frac{4}{\sqrt[n]{\vol_n(\Omega)^{2}}}.
    \end{cases}
   \end{align*}
  Together with Thm.~\ref{thm:satz1}, this gives explicit upper and lower bounds for the remainder term of counting functions (for Neumann, Dirichlet and mixed boundary conditions) of well-covered domains whenever $\delta > n-1$: Since $\vol_n (\Omega_{-\epsilon}) \leq \vol_n \left(\bigcup_{i \in I_\epsilon} D_i ^\epsilon \right) \leq c_{\vol} ^+ C(\Omega) \epsilon^{n-\delta}$, we may set $\widetilde{C} := c_{\vol} ^+ C(\Omega)$ to obtain an estimate for all $t > t_0$ for $|N_N(\Omega,t) - C_W ^{(n)} \vol_n \Omega t^{n/2}| \leq \max\{ M_K ^{\text{abs}},\frac{5\widetilde{C}}{((n-\delta)(\delta+1-n))} \}t^{\delta/2}$ with $M_K ^{\text{abs}}$ taken from \ref{item:allelambda} above. This constant will be denoted by $\widetilde{M_K}$. It plays a central role in the companion paper \cite{DOKUMENT2}. In \cite{DOKUMENT2} these estimates are used to deduce high order asymptotic terms for $N_D(U,t)$ for self-similar sprays $U$ whose generators are finite unions of pairwise disjoint Koch snowflakes.
  \label{item:vandenBerg}
  \item By Thm.~\ref{thm:satz1} there is an $M_\Omega \in \mathcal{O}(1)$ such that $N_N(\Omega,t)= C_W ^{(n)} \vol \Omega t^{n/2} + M_\Omega(t) t^{\delta/2}$. Based on the fact that for any rescaling $\alpha>0$ one has $N_N(\alpha \Omega,t) = N_N(\Omega,\alpha^2 t)$, it is clear that $M_{\alpha \Omega}(t) = \alpha^\delta M_\Omega(\alpha^2t)$.
  \item Combining Thm.~1.10 of \cite{NetrusovSafarov2005} with Thm.~\ref{thm:wellcovered-bv} it follows that for any $n$ there is a well-covered domain $\Omega \subset \mathbb{R}^n$ and a $t_\Omega\in \mathbb{R}_+$ for which $N_N(\Omega,t) - C_W ^{(n)} \vol_n \Omega t^{n/2} \geq t_\Omega ^{-1} t^{\delta/2}$ for all $t> t_\Omega$. In this sense, the result of Thm.~\ref{thm:satz1} is order-sharp.
  \end{enumerate}
  \end{rem}
  \subsection{Application to snowflakes}\label{subsec:application}
  For snowflake domains and more generally any domain whose boundary has a self-similar structure, the values of $\beta(\gamma,t)$ as introduced in Def.~\ref{def:wellfoliated} can often be easily understood: They increase in discrete steps according to a power-law as $t$ approaches $\len \gamma$ (cf.~Fig.~\ref{fig:foliationNEWp}). For the Koch snowflake for example, one might say, that $\beta$ doubles in value each time a fibre passes through $\Phi^k(I_0)$ for a $k\in\mathbb N$.
  \par
  The remainder of this section is devoted to finding an explicit estimate for the remainder term of the eigenvalue counting function of the Neumann Laplacian for a family of snowflakes constructed below.
  \par
  \paragraph{Construction of $p$-snowflakes of Rohde type.} We summarise the construction of a class of snowflakes introduced by Rohde in \cite{rohde2001}. Let $p \in \left(\frac{1}{4},\frac{\sqrt{3}-1}{2}\right)$. For this $p$, we consider the IFS of four contractions, each of contraction ratio $p$, whose action on the unit interval is depicted in Fig.~\ref{fig:Rohde-p}. The limit set of this IFS will be called $p$-Koch curve. Note that the classic Koch curve arises when $p=1/3$. 
   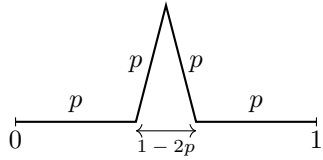
\begin{figure}[h!]
   \centering
     \begin{tikzpicture}[scale=4]
     \def\p{2/5};
     \node[below] at (0,0) {$0$};
     \draw (0,0.015) -- (0,-0.015);
     \node[below] at (1,0) {$1$};
     \draw (1,0.015) -- (1,-0.015);
     \draw[thick] (0,0) -- (\p,0) -- (1/2,{sqrt(3/20)}) -- (1-\p,0)-- (1,0);
     \node[above] at (\p/2,0) {$p$};
     \node[above] at (1-\p/2,0) {$p$};
     \node at ({\p},{sqrt(3/20)/2}) {$p$};
     \node at ({1-\p},{sqrt(3/20)/2}) {$p$};
     \draw[<->] (\p,-0.03) -- (1-\p,-0.03);
     \node[below] at (1/2,-0.03) {\scriptsize$1-2p$};
     \end{tikzpicture}
    \caption{Illustration of the action of the IFS dependent on $p$ on the unit interval, which is used in \cite{rohde2001,goldstein2018}. The resulting curve depicted here shall be denoted by $A_p$.}
     \label{fig:Rohde-p}
   \end{figure}
  Replacing each of the four sides of a unit square with an outwards pointing copy of the $p$-Koch curve 
  produces a snowflake-like domain $\Rohde$. Each thus constructed $\Rohde$ belongs to the class of homogeneous snowflake domains considered by Rohde in \cite{rohde2001}. We will discuss the more general construction of Rohde and its significance in Rem.~\ref{rem:Rohde}. We will also consider snowflakes $\Kochp$, where one replaces each side of an equilateral triangle with an outwards pointing copy of the $p$-Koch curve and note that $\Koch:=\Koch(1/3)$ is the classic Koch snowflake.\par
 \paragraph{Bounds on the Neumann eigenvalue counting function for $p$-snowflakes of Rohde type.}
 Fix $p\in\left(\frac14,\frac{\sqrt{3}-1}{2} \right)$ and let $\delta := \overline{\dim_M} \partial \Rohde= \overline{\dim_M} \partial \Kochp = -\log_p 4$ denote the upper inner Minkowski dimension of the boundary of $\Rohde$ and $\Kochp$. 
  In order to apply Thm.~\ref{thm:satz1}, we need to show that $\Kochp$ and $\Rohde$ are well-covered. To this end, let $\epsilon \in J_k ^{(p)} := \left( p^{k+1}\frac{1-2p}{\sqrt{4p-1}} \right.,\left. p^k \frac{1-2p}{\sqrt{4p-1}}\right]$. Based on a modification of the geometric approximation of $\Koch(1/3)_{-\epsilon}$ in \cite{LapidusPearse}, we construct a cover of $\Kochp_{-\epsilon}$ and $\Rohde_{-\epsilon}$ by well-foliated domains as in the proof of Thm.~\ref{thm:satz1}. This cover has multiplicity $\mu = 2$. Such covers for the classic Koch snowflake and $\Rohde$ are depicted in Fig.~\ref{fig:KochSnowflakeCover} and Fig.~\ref{fig:KochSnowflakeCoverRohde}. 
  The cover consists of three types of objects:
  \begin{enumerate}[label=(\Alph*)]
   \item Fringed rectangles (fR) $(1-2p) p^{k-1} \times \epsilon$ with fractal top of height $p^{k-1}\frac{\sqrt{4p-1}}{2}$. \label{item:1}
   \item Short rectangles (sR) $p^k \times \epsilon$ with a fractal top of height $p^k \frac{\sqrt{4p-1}}{2}$.
   \item Long rectangles (lR) $\left(p^k+\frac{\epsilon\sqrt{4p-1}}{2p}\right)\times \epsilon$ and a fractal top of height $p^k\frac{\sqrt{4p-1}}{2}$.\label{item:3}
  \end{enumerate}

  \begin{figure}[h!]
      \centering
\includegraphics{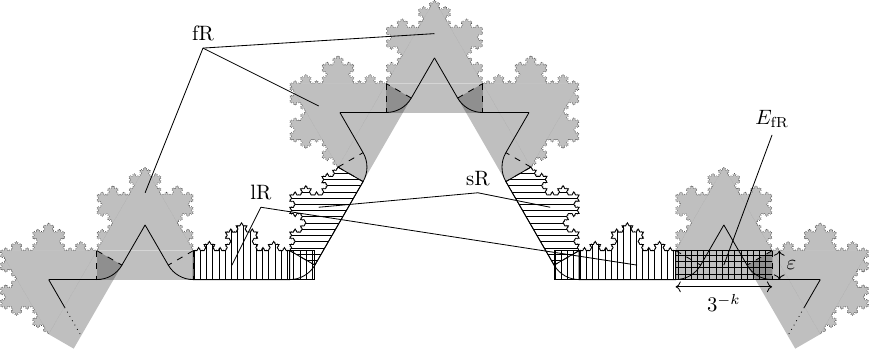}
      \caption{Instance of the cover of an inner $\epsilon$-neighbourhood $K_{-\epsilon}$ of the classic Koch snowflake $K$ and $\epsilon \in J_2 ^{(1/3)}$. There are three types of covering domains used here: fringed Rectangles (fR, coloured gray), short rectangles (sR, horizontal lines) and long rectangles (lR, vertical lines). An example of the underlying domain $E_{\text{fR}}$ is marked with a grid -- the other underlying domains are analogous. Both end segments of this curve are parts of additional fR's, as the snowflake is build out of three equilateral Koch curves. 
      }\label{fig:KochSnowflakeCover}
  \end{figure}
    \begin{figure}[h!]
      \centering
\includegraphics{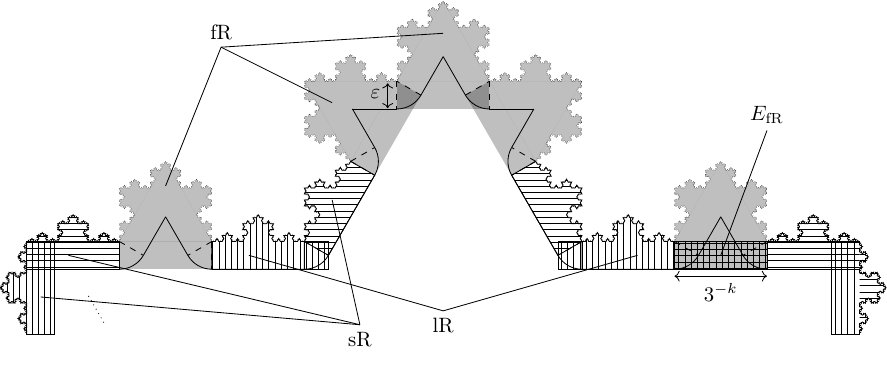}
      \caption{Cover of an instance of $\Rohde$ by domains as in \ref{item:1}-\ref{item:3}. The only difference to the Koch snowflake (see Fig.~\ref{fig:KochSnowflakeCover}) is located at the four edges of the unit square where one now uses two copies of sR instead of one fR. To better illustrate the overlap, the leftmost and rightmost copies of sR are marked by vertical lines instead of horizontal lines.
      }
      \label{fig:KochSnowflakeCoverRohde}
  \end{figure}

    \begin{figure}[h!]
 \centering
 \begin{tikzpicture}[
 declare function = {h(\g) = -sqrt(3)*\g + sqrt(3)/3; h2(\var) = -sqrt(3)/(9*6)*\var + 36*sqrt(3)/(9*6); hx(\var) = -sqrt(3)*\var + sqrt(3)/6; },
			scale=7,decoration=Koch snowflake]
\clip (0.6,-0.3) rectangle (2.6,0.9);
    \draw[line width=0.01mm] decorate{decorate{decorate{ decorate{ decorate{ decorate{ (0,0) -- (3,0) }}}}}};
    \draw[fill=black!20!white!40!,line width=0.01mm,dashed] decorate{ (0,0) -- (3,0) };
    \draw[white] (0,0) -- (3,0);
    \draw[fill=black!80!white!40!] decorate{decorate{decorate{ decorate{(1,0) -- (3/2,0.866025) }}}};
    \draw[fill=black!80!white!40!] decorate{decorate{decorate{ decorate{(3-3/2,0.866025) -- (3-1,0) }}}};
    \draw[fill=black!40!white!40!,line width=0.01mm,dashed] decorate{ decorate{(1,0) -- (3/2,0.866025) }};
    \draw[fill=black!40!white!40!,line width=0.01mm,dashed] decorate{ decorate{(3-3/2,0.866025) -- (3-1,0) }};
    \draw[fill=black!40!white!40!,line width=0.01mm] decorate{decorate{ (1,0) -- (1+1/18,0.866025/9) }};
    \draw[fill=black!40!white!40!,line width=0.01mm] decorate{decorate{ (3-1-1/18,0.866025/9) -- (3-1,0) }};
    \draw[fill=black!40!white!40!,line width=0.01mm] decorate{decorate{ (1+0.111111,0.19245) -- (1+1/18+0.111111,0.866025/9+0.19245) }};
    \draw[fill=black!40!white!40!,line width=0.01mm] decorate{decorate{ (3-1-1/18-0.111111,0.866025/9+0.19245) -- (3-1-0.111111,0.19245) }};
    \draw[fill=black!40!white!40!,line width=0.01mm] decorate{decorate{ (1+0.333333, 0.57735) -- (1+1/18+0.333333,0.866025/9+0.57735) }};
    \draw[fill=black!40!white!40!,line width=0.01mm] decorate{decorate{ (3-1-1/18-0.333333,0.866025/9+0.57735) -- (3-1-0.333333, 0.57735) }};
    \draw[fill=black!40!white!40!,line width=0.01mm] decorate{decorate{ (1+0.444444, 0.7698) -- (1+1/18+0.444444,0.866025/9+0.7698) }};
    \draw[fill=black!40!white!40!,line width=0.01mm] decorate{decorate{ (3-1-1/18-0.444444,0.866025/9+0.7698) -- (3-1-0.444444, 0.7698)}};
    \draw[fill=black!80!white!40!,line width=0.01mm] decorate{decorate{ decorate{ (1.16667, 0.288675) -- (1, 0.57735) }}};
    \draw[fill=black!80!white!40!,line width=0.01mm] decorate{decorate{ decorate{ (3-1, 0.57735) -- (3-1.16667, 0.288675)}}};
    \draw[fill=black!80!white!40!,line width=0.01mm] decorate{decorate{ decorate{ (1, 0.57735) -- (1.33333, 0.57735) }}};
    \draw[fill=black!80!white!40!,line width=0.01mm] decorate{decorate{ decorate{ (3-1.33333, 0.57735) -- (3-1, 0.57735) }}};%
    \draw[fill=black!120!white!40!,line width=0.01mm] decorate{decorate{ decorate{ (1.11111, 0.3849) -- (1, 0.3849) }}};
    \draw[fill=black!120!white!40!,line width=0.01mm] decorate{decorate{ decorate{ (3-1, 0.3849) -- (3-1.11111, 0.3849)}}};
    \draw[fill=black!120!white!40!,line width=0.01mm] decorate{decorate{ decorate{ (1, 0.3849) -- (1.05556, 0.481125) }}};
    \draw[fill=black!120!white!40!,line width=0.01mm] decorate{decorate{ decorate{ (3-1.05556, 0.481125) -- (3-1, 0.3849) }}};
    \draw[fill=black!120!white!40!,line width=0.01mm] decorate{decorate{ decorate{ (1+0.0555556+0.0555556, 0.3849+0.096225+0.096225) -- (1.05556+0.0555556+0.0555556, 0.481125+0.096225+0.096225) }}};
    \draw[fill=black!120!white!40!,line width=0.01mm] decorate{decorate{ decorate{ (3-1.05556-0.0555556-0.0555556, 0.481125+0.096225+0.096225) -- (3-1-0.0555556-0.0555556, 0.3849+0.096225+0.096225) }}};
    \draw[fill=black!120!white!40!,line width=0.01mm] decorate{decorate{ decorate{ (1.05556+0.0555556+0.0555556, 0.481125+0.096225+0.096225) -- (1.05556+0.0555556+0.0555556+0.0555556, 0.481125+0.096225+0.096225-0.096225)}}};
    \draw[fill=black!120!white!40!,line width=0.01mm] decorate{decorate{ decorate{ (3-1.05556-0.0555556-0.0555556-0.0555556, 0.481125+0.096225+0.096225-0.096225) -- (3-1.05556-0.0555556-0.0555556, 0.481125+0.096225+0.096225)}}};
	\draw[pattern=dots]
    (1,0) -- (2,0) -- (2,-0.2) -- (1,-0.2) -- (1,0);
	\node at (1.35,0.35) {\tiny$\beta=1$};
	\node at (1.15,0.525) {\tiny$\beta=\frac{2p}{1-2p}$};
	\node[left] at (0.85,0.425) {\tiny$\beta=\left(\frac{2p}{1-2p}\right)^2$};\draw[line width=0.05mm] (0.85,0.425) -- (1.05,0.415);
	\node[left] at (0.85,0.325) {\tiny$\beta=\left(\frac{2p}{1-2p}\right)^3$};\draw[line width=0.05mm] (0.85,0.325) -- (1.05,0.375);
	\node[scale=1] at (1.2,-0.1) {$E$};
	\draw[line width=0.4mm] (1,-0.2) -- (2,-0.2);
	\node[scale=1,above] at (1.1,-0.2) {$Q$};
	\node[scale=1] at (1.4,0.065) {$D$};
    \draw[<->] (2.025,0) -- (2.025,-0.2);
    \node[rotate=90,below] at (2.025,-0.1) {$r=\epsilon$};
    \draw[<->] (1,-0.225) -- (2,-0.225);
    \node[below] at (1.5,-0.225) {$(1-2p)p^{k-1}$};
	\foreach \k in {0,1,2,4,5,6,7,8,9,10,11,12,13,14,15,16,17,18,19,20,21,22,23,24,25,26,27,28,29,30,31,33,34,35,36,37,38,39,40,41,42,43,44,45,46,47,48,49,50,51,52,53,54}{.
        \draw (1.5+\k/108,-0.2) -- (1.5+\k/108,{-sqrt(3)*(1.5+\k/108)+2*sqrt(3)});
        }
        \foreach \k in {32}{
        \draw (1.5+\k/108,-0.2) -- (1.5+\k/108,{-sqrt(3)*(1.5+\k/108)+2*sqrt(3)});
        }
        \foreach \k in {1}{
        \draw (1.5+\k/108,{-sqrt(3)*(1.5+\k/108)+2*sqrt(3)}) -- ({1.5+\k/108+sqrt(3)/2*(1/27)*(h2(27+1-\k))},{-sqrt(3)*(1.5+\k/108)+2*sqrt(3)+1/2*(1/27)*(h2(27-1+\k))});
        }
        \foreach \k in {3}{
        \draw (1.5+\k/108,{-sqrt(3)*(1.5+\k/108)+2*sqrt(3)}) -- ({1.5+\k/108+sqrt(3)/2*(1/3)*(1/3)*(h2(27+3-\k))},{-sqrt(3)*(1.5+\k/108)+2*sqrt(3)+1/2*(1/3)*(1/3)*(h2(27-3+\k))});
        }
        \foreach \k in {5}{
        \draw (1.5+\k/108,{-sqrt(3)*(1.5+\k/108)+2*sqrt(3)}) -- ({1.5+\k/108+sqrt(3)/2*(1/27)*(h2(27+5-\k))},{-sqrt(3)*(1.5+\k/108)+2*sqrt(3)+1/2*(1/27)*(h2(27-5+\k))});
        }
        \foreach \k in {6,7,8,9}{
        \draw (1.5+\k/108,{-sqrt(3)*(1.5+\k/108)+2*sqrt(3)}) -- ({1.5+\k/108+sqrt(3)/2*(1/3)*(h2(54-3*\k))},{-sqrt(3)*(1.5+\k/108)+2*sqrt(3)+1/2*(1/3)*(h2(54-3*\k))});
        }
        \foreach \k in {9,10,11,12}{
        \draw (1.5+\k/108,{-sqrt(3)*(1.5+\k/108)+2*sqrt(3)}) -- ({1.5+\k/108+sqrt(3)/2*(1/3)*(h2(3*\k))},{-sqrt(3)*(1.5+\k/108)+2*sqrt(3)+1/2*(1/3)*(h2(3*\k))});
        }
        \foreach \k in {13}{
        \draw (1.5+\k/108,{-sqrt(3)*(1.5+\k/108)+2*sqrt(3)}) -- ({1.5+\k/108+sqrt(3)/2*(1/27)*(h2(27))},{-sqrt(3)*(1.5+\k/108)+2*sqrt(3)+1/2*(1/27)*(h2(27))});
        }
        \foreach \k in {15}{
        \draw (1.5+\k/108,{-sqrt(3)*(1.5+\k/108)+2*sqrt(3)}) -- ({1.5+\k/108+sqrt(3)/2*(1/3)*(1/3)*(h2(27+15-\k))},{-sqrt(3)*(1.5+\k/108)+2*sqrt(3)+1/2*(1/3)*(1/3)*(h2(27-15+\k))});
        }
        \foreach \k in {17}{
        \draw (1.5+\k/108,{-sqrt(3)*(1.5+\k/108)+2*sqrt(3)}) -- ({1.5+\k/108+sqrt(3)/2*(1/27)*(h2(27))},{-sqrt(3)*(1.5+\k/108)+2*sqrt(3)+1/2*(1/27)*(h2(27))});
        }
        \foreach \k in {18,19,20,21,22,23,24,25,26,27}{
        \draw (1.5+\k/108,{-sqrt(3)*(1.5+\k/108)+2*sqrt(3)}) -- ({1.5+\k/108+sqrt(3)/2*(h2(54-\k))},{-sqrt(3)*(1.5+\k/108)+2*sqrt(3)+1/2*(h2(54-\k))});
        }
        \foreach \k in {22,23}{
        \draw ({1.5+\k/108+sqrt(3)/2*(h2(54-\k))},{-sqrt(3)*(1.5+\k/108)+2*sqrt(3)+1/2*(h2(54-\k))}) -- ({1.5+\k/108+sqrt(3)/2*(h2(54-\k))+0},{-sqrt(3)*(1.5+\k/108)+2*sqrt(3)+1/2*(h2(54-\k))+1/9*sqrt(3)/3});
        }
	\foreach \k in {27,28,29,30,31,33,34,35,36}{
        \draw (1.5+\k/108,{-sqrt(3)*(1.5+\k/108)+2*sqrt(3)}) -- ({1.5+\k/108+sqrt(3)/2*(h2(\k))},{-sqrt(3)*(1.5+\k/108)+2*sqrt(3)+1/2*(h2(\k))});
        }
        \foreach \k in {32}{
        \draw (1.5+\k/108,{-sqrt(3)*(1.5+\k/108)+2*sqrt(3)}) -- ({1.5+\k/108+sqrt(3)/2*(h2(\k))},{-sqrt(3)*(1.5+\k/108)+2*sqrt(3)+1/2*(h2(\k))});
        }
        \foreach \k in {31}{
        \draw ({1.5+\k/108+sqrt(3)/2*(h2(\k))},{-sqrt(3)*(1.5+\k/108)+2*sqrt(3)+1/2*(h2(\k))}) -- ({1.5+\k/108+sqrt(3)/2*(h2(\k))+sqrt(3)/2*1/9*sqrt(3)/3},{-sqrt(3)*(1.5+\k/108)+2*sqrt(3)+1/2*(h2(\k))-1/2*1/9*sqrt(3)/3});
        }
        \foreach \k in {32}{
        \draw ({1.5+\k/108+sqrt(3)/2*(h2(\k))},{-sqrt(3)*(1.5+\k/108)+2*sqrt(3)+1/2*(h2(\k))}) -- ({1.5+\k/108+sqrt(3)/2*(h2(\k))+sqrt(3)/2*1/9*sqrt(3)/3},{-sqrt(3)*(1.5+\k/108)+2*sqrt(3)+1/2*(h2(\k))-1/2*1/9*sqrt(3)/3});
        }
        \foreach \k in {37}{
        \draw (1.5+\k/108,{-sqrt(3)*(1.5+\k/108)+2*sqrt(3)}) -- ({1.5+\k/108+sqrt(3)/2*(1/27)*(h2(27))},{-sqrt(3)*(1.5+\k/108)+2*sqrt(3)+1/2*(1/27)*(h2(27))});
        }
        \foreach \k in {39}{
        \draw (1.5+\k/108,{-sqrt(3)*(1.5+\k/108)+2*sqrt(3)}) -- ({1.5+\k/108+sqrt(3)/2*(1/3)*(1/3)*(h2(27+39-\k))},{-sqrt(3)*(1.5+\k/108)+2*sqrt(3)+1/2*(1/3)*(1/3)*(h2(27-39+\k))});
        }
        \foreach \k in {41}{
        \draw (1.5+\k/108,{-sqrt(3)*(1.5+\k/108)+2*sqrt(3)}) -- ({1.5+\k/108+sqrt(3)/2*(1/27)*(h2(27))},{-sqrt(3)*(1.5+\k/108)+2*sqrt(3)+1/2*(1/27)*(h2(27))});
        }
        \foreach \k in {42,43,44,45}{
        \draw (1.5+\k/108,{-sqrt(3)*(1.5+\k/108)+2*sqrt(3)}) -- ({1.5+\k/108+sqrt(3)/2*(1/3)*(h2(162-3*\k))},{-sqrt(3)*(1.5+\k/108)+2*sqrt(3)+1/2*(1/3)*(h2(162-3*\k))});
        }
        \foreach \k in {45,46,47,48}{
        \draw (1.5+\k/108,{-sqrt(3)*(1.5+\k/108)+2*sqrt(3)}) -- ({1.5+\k/108+sqrt(3)/2*(1/3)*(h2(3*\k-108))},{-sqrt(3)*(1.5+\k/108)+2*sqrt(3)+1/2*(1/3)*(h2(3*\k-108))});
        }
        \foreach \k in {49}{
        \draw (1.5+\k/108,{-sqrt(3)*(1.5+\k/108)+2*sqrt(3)}) -- ({1.5+\k/108+sqrt(3)/2*(1/27)*(h2(27))},{-sqrt(3)*(1.5+\k/108)+2*sqrt(3)+1/2*(1/27)*(h2(27))});
        }
        \foreach \k in {51}{
        \draw (1.5+\k/108,{-sqrt(3)*(1.5+\k/108)+2*sqrt(3)}) -- ({1.5+\k/108+sqrt(3)/2*(1/3)*(1/3)*(h2(27+51-\k))},{-sqrt(3)*(1.5+\k/108)+2*sqrt(3)+1/2*(1/3)*(1/3)*(h2(27-51+\k))});
        }
        \foreach \k in {53}{
        \draw (1.5+\k/108,{-sqrt(3)*(1.5+\k/108)+2*sqrt(3)}) -- ({1.5+\k/108+sqrt(3)/2*(1/27)*(h2(27))},{-sqrt(3)*(1.5+\k/108)+2*sqrt(3)+1/2*(1/27)*(h2(27))});
        }
	\foreach \k in {3}{
        \draw (1.5+\k/108,-0.2) -- (1.5+\k/108,{-sqrt(3)*(1.5+\k/108)+2*sqrt(3)});
        }
\end{tikzpicture}
\caption{Systematic construction of the foliation of sets near the boundary of the $p$-snowflake of Rohde type in the case of a \textquotedblleft{}fringed rectangle (fR)\textquotedblright{} (see Sec.~\ref{subsec:application}\ref{item:1}). The well-foliated domain is the dotted rectangle (the set $E$) and the section of the snowflake above it. The gray areas within the snowflake indicate areas of constant $\beta(\gamma,t)$. One notices that in the particular case of $p=1/3$, a connection between the ternary expansion of $\gamma(0) \in \mathbb{R}$ and the shape of the curve (and in particular its number of turns) is evident.}\label{fig:foliationNEWp}
\end{figure}
We will now show that the fringed rectangles are well-foliated. Analoguously it can be shown that the short rectangles and the long rectangles are well-foliated, too. The foliation $\Gamma$, as constructed in Sec.~\ref{subsec:foliation}, is shown in Fig.~\ref{fig:foliationNEWp}. For every $\gamma =\gamma_q \in \Gamma$ with $q\in Q$ the integral $\int_0 ^{\len \gamma} \beta(\gamma,t)dt$ is bounded by the sum of integrals over $\beta(\gamma,t)$ along the longest fibre within each iteration of triangles. More precisely, we have a partition $0=t_{-1} < t_0 < t_1 \cdots$ of $[0,\len \gamma]$ such that $\gamma_q(t_{\ell})\in\Phi^{\ell}(I_0)$, yielding
\begin{align}
\begin{aligned}\label{eq:beta}
 \int_0 ^{\len \gamma} \beta(\gamma,t) dt
 = \sum_{\ell=0}^{\infty} \int_{t_{\ell-1}} ^{t_{\ell}} \beta(\gamma,t) dt
 &< r + \sum_{\ell=1}^{\infty} \frac{\sqrt{4p-1}}{2}\cdot p^{k+\ell-2} \left( \frac{2p}{1-2p} \right)^{\ell-1}
 \\
 &= r + p^{k-1} \frac{(1-2p)\sqrt{4p-1}}{2(1-2p-2p^2)} < \infty.
 \end{aligned}
\end{align}
With $r=\epsilon$, this implies that $\mathcal{I}_\beta < \epsilon \left( 1+\frac{4p-1}{2p^2(1-2p-2p^2)} \right)$. Moreover, in all of the cases of covering domains (fR, sR and lR), the foliation is constructed as shown in Fig.~\ref{fig:foliationNEWp} so that $\inf \beta = 1$, showing that fR, sR and lR are well-foliated. It also provides an intuition of how to construct a suitable family of foliations in similar cases.\par
To show that a snowflake as constructed above is a well-covered domain, we now need to check parts \ref{item:zweidef:1}-\ref{item:zweidef:4} of Def.~\ref{def:wellcovered}.\par
  To \ref{item:zweidef:1}: By the iteration dynamics, $\Rohde_{-\epsilon}$ is covered by $\frac{4}{3}(4^k - 1)$ instances of fR and $\frac{4}{3}(4^k+2)$ copies of sR or lR. In total this amounts to a cover of cardinality $\#I_\epsilon = \frac{4}{3}(2\cdot 4^k +1) \leq C(\Rohde) \epsilon^{-\delta}$ with $C(\Rohde) = 3\left( \frac{1-2p}{\sqrt{4p-1}} \right)^{\delta}$. Analogously one can chose $C(\Kochp) = \frac{3}{4}C(\Rohde)$. For the classic Koch snowflake $\Koch:=\Koch(1/3)$, the cardinality of the cover is $2\cdot 4^k-2 \leq C(\Koch) \epsilon^{-\delta}$ with $C(\Koch)=1$.\par
  To \ref{item:zweidef}: By construction $r=\epsilon$, so that $c_r^{\pm}=1$. The length $L$ of a fibre satisfies $\epsilon=r\leq L< r+\sum_{\ell=0}^{\infty} p^{\ell+k-1}\frac{\sqrt{4p-1}}{2} \leq \left( 1+\frac{4p-1}{p^2(2-6p+4p^2)} \right)\epsilon$ implying $c_L^-\geq1$ and $c_L^+\leq 1+\frac{4p-1}{p^2(2-6p+4p^2)}$. From our previous estimate on $\mathcal I_{\beta}$ we infer $c_{\mathcal I}^-\geq 1$ and $c_{\mathcal I}^+\leq 1+ \frac{4p-1}{2p^2(1-2p-2p^2)}$.
   The diameters of the covering domains \ref{item:1}-\ref{item:3} above are found to be bounded from above by $\sqrt{[(1-2p)p^{k-1}]^2 + \left(\epsilon + p^{k-1}\frac{\sqrt{4p-1}}{2}\right)^2}$ so that $c_{\diam} ^+ \leq \sqrt{\frac{4p-1}{p^4} + \left( 1-\frac{4p-1}{2p^2(1-2p)} \right)^2}$. Moreover, $c_{\diam} ^-\geq 1$.\par
   To \ref{item:zweidef:3} and \ref{item:zweidef:4}: By minimising $\lambda_2 ^N(E_{\text{lR}})\epsilon^2$ as a function of $\epsilon \in J_k ^{(p)}$ we obtain    $c_E(\Rohde) =\min\left(1, \frac{4(1-2p)^2p^2}{(3-2p)^2(4p-1)} \right)\pi^2$. 
   Using the above constants we compute $C_1(\Rohde_{-\epsilon})$ as given in Cor.\ref{cor:ewestimate} so that 
   $\lambda_2 ^N(D_i ^\epsilon) \geq \epsilon^{-2} C_1(\Rohde_{-\epsilon})$.
  See Fig.~\subref{fig:plotC1} for a plot of $C_1(\Rohde_{-\epsilon})$.
  \begin{figure}[h!]
    \centering
    \begin{subfigure}[t]{0.45\textwidth}
        \centering
        \includegraphics[width=\textwidth]{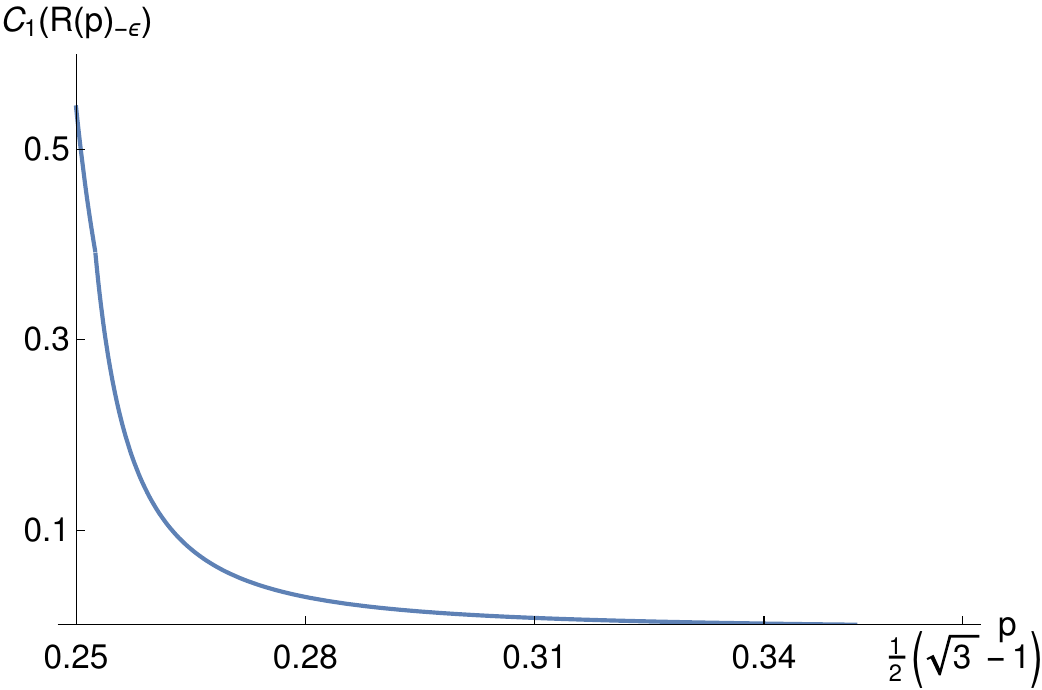}
        \subcaption{Values of $C_1(\Rohde_{-\epsilon})$ for values for $p \in (\frac{1}{4},\frac{\sqrt{3}-1}{2})$ as in Cor.\ref{cor:ewestimate}. The values approach $0$ as $p \nearrow \frac{\sqrt{3}-1}{2}$.}\label{fig:plotC1}
    \end{subfigure}%
    \qquad
    \begin{subfigure}[t]{0.45\textwidth}
        \centering
        \includegraphics[width=\textwidth]{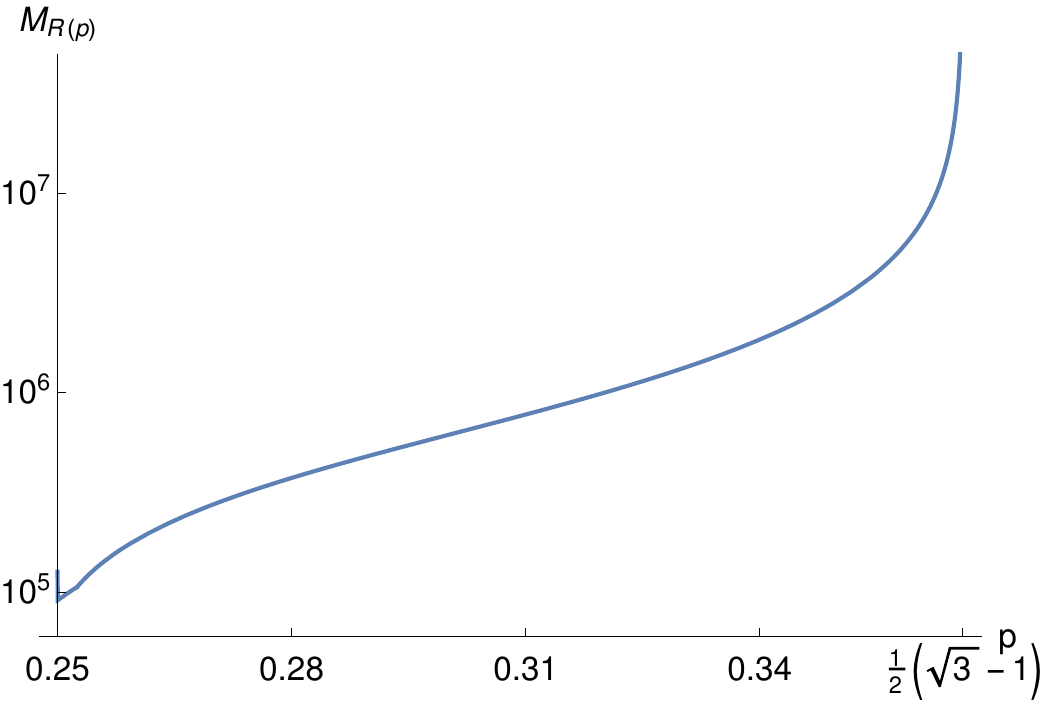}
        \caption{Values of $M_{K(p)}$ for values for $p \in (\frac{1}{4},\frac{\sqrt{3}-1}{2})$ as given by Thm.~\ref{thm:satz1}.}
      \label{fig:plotC4}
    \end{subfigure}
\end{figure}
  In particular in the case of the classic Koch snowflake one finds $\lambda_2 ^N(D) \geq 0,0031\epsilon^{-2}$ for any element of a cover of $K_{-\epsilon}$ as given by \ref{item:1}-\ref{item:3} above.\par
  Based on the above considerations, one now obtains expressions for $M_{\Rohde}$ as well as coefficients following Rem.~\ref{rem:countingfunctions}.\ref{item:allelambda}. As expected, the coefficients $M_{\Rohde}$ diverge as $p \nearrow \frac{\sqrt{3}-1}{2}$ and remain regular as $p \searrow 1/4$, see Fig.~\subref{fig:plotC4}.\par
  In the case of the classic Koch snowflake $K$, a slighty better estimate is possible based on an estimate of the inner tube area by Lapidus and Pearse in \cite{LapidusPearse}:
  \begin{align*}
   \vol_2(K_{-\epsilon}) \leq 3\left[\epsilon^{2-\delta} 4^{-\{x\}}\left( \frac{3\sqrt{3}}{40}9^{\{x\}} + \frac{\sqrt{3}}{2}3^{\{x\}} +\frac{1}{6}\left( \frac{\pi}{3}-\sqrt{3} \right) \right) - \frac{\epsilon^2}{3}\left( \frac{\pi}{3} + 2\sqrt{3} \right)\right],
  \end{align*}
  where $\delta = \log_3 4 = \overline{\dim_M}(\partial K,K)$ and $\{x\}$ is the fractional part of $x:=-\log_3 (\epsilon \sqrt{3})$.
  Therefore,
  \begin{align*}
   \overline{\mathcal{M}}_\delta(\partial K,K) \leq \frac{1}{480} \left( 723\sqrt{3} + 20\pi \right).
  \end{align*}
  Moreover, the monotonicity of the above upper bound shows that $\epsilon' \leq 0$ in our estimate and hence $\mathfrak{M}_K \leq 11.61$.
  Based on this, $C_3(K) \approx 1354$ and $S_1 ^\epsilon(t) \leq 104282t^{\delta/2}$. Computing $M_{K}$ we finally obtain an asymptotic upper bound for the remainder term of the Neumann counting function of the Koch snowflake of 
  \[N_N(K,t)-C_W ^{(n)}\vol_2(K)t \leq 104325{.}5t^{\delta/2}.\]
  Alternatively, an absolute upper bound can be found following Rem.~\ref{rem:countingfunctions}.\ref{item:allelambda}:
  \begin{align*}
      N_N(K,t) - C_W ^{(2)} \vol(K)t \leq C_W ^{(2)}\left(3.537 \cdot 10^6t^{\delta/2} -353 - 911\sqrt{t}\right).
  \end{align*}
  Since $\epsilon < 1/9$ is more than sufficient to satisfy Def.~\ref{def:wellcovered}, this holds true for all $t \geq 0.1$. Some limited computational results on the Neumann (or Dirichlet) spectrum of the Koch Snowflake can be found in \cite{NeubergerSiebenSwift2006,stwi2020}. A few estimations in the above application were deliberately chosen sub-optimally for the benefit of the presentation so that some improvements in these constants are expected to be possible.
\begin{rem}\label{rem:Rohde}
    Fix $p\in\left(\frac14,\frac12\right)$. By replacing each of the four sides of a unit square $S_0$ with either an outwards pointing copy of the curve $A_p$ depicted in Fig.~\ref{fig:Rohde-p}  or with $A_{1/4}$ one arrives at an object $S_1$ which is made up of $4\cdot 4$ intervals of lengths $p$ or $\frac14$. Applying either substitution individually to each of the resulting intervals, and iterating the procedure leads to an object $S_n$ which is made up of $4\cdot 4^n$ intervals. This construction is presented in \cite{rohde2001} and we will refer to the limiting objects as  \emph{general snowflake-domains of Rohde type}.
    For such snowflake-domains, estimates for Poincar\'{e} constants were obtained in \cite{goldstein2018}.
    The snowflake-domains $\Rohde$ that we consider in the present section arise as the special cases where the curve $A_p$ is chosen in every replacement step.
    \par
    Rohde showed in \cite{rohde2001} that the general snowflake-domains of Rohde type form an exhaustive family of quasicircles up to bi-Lipschitz maps. Besides a more involved computation of the upper inner Minkowski dimension and upper inner Minkowski content of the boundary,   the construction of a foliation and the arguments applied in the present section carry over to the general snowflake-domains of Rohde type for $p\in\left( \frac14,\frac{\sqrt{3}-1}{2} \right)$. The restriction $p<\frac{\sqrt{3}-1}{2}$ ensures that $\int_0^{\len\gamma} \beta(\gamma,t)dt<\infty$ in \eqref{eq:beta}. Since the above argument for Lem.~\ref{lem:ewestimate} and hence Thm.~\ref{thm:satz1} is not affected by bi-Lipschitz transformations (Rem.~\ref{rem:question}), knowledge of the upper inner Minkowski content would lead to estimates for the Neumann counting function for a large family of quasicircular domains. 
\end{rem}
  
  \section{Constants}\label{sec:constants}
  The following constants are used within this document. $\mathfrak{M}_{\Omega}$ and $A_{\Omega}$ are defined in Prop.~\ref{prop:whitneycardinality} and Prop.~\ref{prop:whitneyumfang}, respectively.
  \begin{align*}
   C_1(\Omega_{-\epsilon} ) 
   &:=\left[ c_E ^{-1} (c_r ^+)^2 \left(1+\frac{c_r ^+ + \sqrt{c_E c_L ^+ c_r ^-}}{c_r ^+ \inf_{i \in I_\epsilon} \beta_{\inf}} \frac{c_\mathcal{I} ^+}{c_r ^-} \right) + \frac{1+\frac{c_r ^+}{\sqrt{c_E c_L ^+ c_r ^-}}}{\inf_{i \in I_\epsilon} \beta_{\inf}} c_L ^+ c_{\mathcal{I}} ^+ \right]^{-1}\\
   C_2(\Omega) &:= \min \left\{ \inf_{\epsilon>0} C_1(\Omega_{-\epsilon}), \frac{1}{2}\left( \sqrt{n}\pi/c_{\diam} ^+ \right)^2 \right\}\\
   C_3(\Omega) &:= C(\Omega) + {\mathfrak{M}_{\Omega}} \frac{(40\sqrt{n})^\delta }{2^\delta-1} \\
   M_{\Omega} &:= C_3(\Omega)\left( \frac{\mu+1}{C_2(\Omega)}\right)^{\delta/2} + \frac{C_W ^{(n-1)}}{4}A_\Omega \left( \frac{\mu+1}{C_2(\Omega)}\right)^{\frac{\delta-(n-1)}{2}}.
  \end{align*}
 \bibliographystyle{amsplain}
 \providecommand{\bysame}{\leavevmode\hbox to3em{\hrulefill}\thinspace}
\providecommand{\MR}{\relax\ifhmode\unskip\space\fi MR }

\end{document}